\documentclass[reqno,letterpaper,11pt]{article}

\usepackage{hyperref}

\usepackage{times}
\usepackage[mathcal]{euscript}

\usepackage{amsmath}
\usepackage{amsfonts,amssymb,amsmath,latexsym,wasysym,mathrsfs,bbm,stmaryrd}
\usepackage[all]{xy}
\usepackage[usenames]{color}
\usepackage{epsfig}

\usepackage{amsthm}
\usepackage{thmtools}
\usepackage{thm-restate}
\declaretheorem[numberwithin=section]{theorem}
\declaretheorem[sibling=theorem]{proposition}
\declaretheorem[sibling=theorem]{definition}
\declaretheorem[sibling=theorem]{corollary}
\declaretheorem[sibling=theorem]{lemma}

\declaretheorem[sibling=theorem]{notation}
\declaretheorem[sibling=theorem,name=Theorem / Definition]{thmdef}
\declaretheorem[sibling=theorem,style=remark]{remark}
\declaretheorem[sibling=theorem,style=remark]{example}

\numberwithin{equation}{section} 

\def\R{\mathbb R}
\def\C{\mathbb C}
\def\Z{\mathbb Z}
\def\N{\mathbb N}
\def\T{\mathbb U}
\def\F{\mathbb F}

\def\P{\mathbb P}
\def\E{\mathbb E}
\def\Var{\mathrm{Var}}
\def\Cov{\mathrm{Cov}}
\def\u{\mathfrak{u}}
\def\gl{\mathfrak{gl}}
\def\deg{\mathrm{deg}}
\def\nor{\mathrm{nor}}

\def\Lip{\mathrm{Lip}}

\def\1{\mathbbm 1}

\def\e{\epsilon}
\def\d{\delta}

\def\p{\varphi}

\def\ex{\varepsilon}

\def\t{\tau}

\def\M{\mathbb{M}}
\def\U{\mathbb{U}}
\def\GL{\mathbb{GL}}
\def\A{\mathscr{A}}
\def\EX{\mathscr{E}}
\def\PP{\mathscr{P}}
\def\D{\mathcal{D}}
\def\L{\mathcal{L}}
\def\CC{\mathcal{C}}
\def\B{\mathscr{B}}
\def\u{\mathfrak{u}}
\def\H{\mathcal{H}}

\def\supp{\mathrm{supp}\,}
\def\del{\partial}

\newcommand{\tr}{\mathrm{tr}}
\newcommand{\Tr}{\mathrm{Tr}}
\newcommand{\mx}[1]{\mathbf{#1}}

\renewcommand\emptyset{\varnothing}

\begin{document}

\title{Heat Kernel Empirical Laws on $\U_N$ and $\GL_N$}
\author{Todd Kemp\thanks{Supported by NSF CAREER Award DMS-1254807} \\
Department of Mathematics \\
University of California, San Diego \\
La Jolla, CA 92093-0112 \\
\texttt{tkemp@math.ucsd.edu}
}

\date{\today}

\maketitle

\begin{abstract} This paper studies the empirical measures of eigenvalues and singular values for random matrices drawn from the heat kernel measures on the unitary groups $\U_N$ and the general linear groups $\GL_N$, for $N\in\N$.  It establishes the strongest known convergence results for the empirical eigenvalues in the $\U_N$ case, and the first known almost sure convergence results for the eigenvalues and singular values in the $\GL_N$ case.  The limit noncommutative distribution associated to the heat kernel measure on $\GL_N$ is identified as the projection of a flow on an infinite-dimensional polynomial space.  These results are then strengthened from variance estimates to $L^p$ estimates for even integers $p$.  \end{abstract}

\tableofcontents

\section{Introduction\label{Section Introduction}}

This paper is concerned with the empirical eigenvalue measures associated to heat kernels on the unitary groups and the general linear groups.  Let $\M_N$ denote $N\times N$ complex matrices, let $\U_N = \{U\in\M_N\colon UU^\ast=I_N\}$ be the unitary group, and $\GL_N\subset\M_N$ the general linear group of invertible $N\times N$ matrices.  The unitary group $\U_N$ is a real Lie group, and $\GL_N$ is its complexification.  These Lie groups possess natural Laplace operators $\Delta_{\U_N}$ and $\Delta_{\GL_N}$; cf.\ Definition \ref{d.Laplace} below.  The {\em heat kernel} $\rho^N_t$ is the fundamental solution to the heat equation $\del_t\rho^N_t = \frac12\Delta_{\U_N}\rho^N_t$ on $\U_N$; similarly the heat kernel $\mu^N_t$ is the fundamental solution to the heat equation $\del_t\mu^N_t = \frac12\Delta_{\GL_N}\mu^N_t$ on $\GL_N$.  They are strictly positive smooth probability densities with respect to the (right) Haar measures, and so we identify each density with its measure when convenient.  In fact, we will consider a two-parameter heat kernel $\mu^N_{s,t}$ on $\GL_N$, where $s,t>0$ and $s>t/2$, which interpolates between $\rho^N_s$ when $t=0$ and $\mu^N_{t/2}$ when $s=t$; cf.\ Definition \ref{d.Laplace}.

To fix notation, for $N\in\N$ and $s,t>0$ with $s>t/2$, we set
\begin{align*} &U^N_t \text{ is a random unitary matrix with joint law of entries } \rho^N_t, \text{ and} \\
&Z^N_{s,t} \text{ is a random invertible matrix with joint law of entries } \mu^N_{s,t}.
\end{align*}
Let $(\Omega,\mathscr{F},\P)$ be a probability space from which all the random matrices $\{U^N_t,Z^N_{s,t};N\in\N,s,t>0,s>t/2\}$ are sampled.  As usual, for $F\in L^1(\Omega,\mathscr{F},\P)$, denote $\E(F) = \int_\Omega F\,d\P$.

\subsection{Main Theorems}

We are interested in the {\bf empirical eigenvalue measures} of these matrices.  For $Z\in\M_N$ denote by $\Lambda(Z)$ the unordered list of eigenvalues of $Z$, counted with multiplicities. The empirical eigenvalue measures are the following random discrete measures on $\C$:
\begin{equation} \label{e.defempphi} \widetilde{\nu}^N_t = \frac1N\sum_{\lambda\in\Lambda(U^N_t)} \delta_{\lambda} \qquad \text{and} \qquad \widetilde{\phi}^N_{s,t} = \frac1N\sum_{\lambda\in\Lambda(Z^N_{s,t})} \delta_{\lambda}. \end{equation}
To describe the limit behavior of these random measures, we introduce the following one-parameter family of probability measures.

\begin{thmdef} \label{d.nuNew} For each $t\in\R$, there exists a unique probability measure $\nu_t$ on $\C^\ast=\C\setminus\{0\}$ with the following properties.  For $t>0$, $\nu_t$ is supported in the unit circle $\U$; for $t<0$, $\nu_t$ is supported in $\R_+=(0,\infty)$; and $\nu_0=\delta_1$.  In all cases, $\nu_t$ is determined by its moments: $\nu_0(t)\equiv 1$ and, for $n\in\Z\setminus\{0\}$,
\begin{equation} \label{e.nuNu0} \nu_n(t) \equiv \int_{\C^\ast} u^n\,\nu_t(du) = e^{-\frac{|n|}{2}t}\sum_{k=0}^{|n|-1} \frac{(-t)^k}{k!}|n|^{k-1}\binom{|n|}{k+1}. \end{equation}
For all $t\ne 0$, $\nu_t$ possesses a continuous density $\varrho_t$ with connected, compact support; $\varrho_t$ is strictly positive in a neighborhood of $1$ (in $\U$ for $t>0$, in $\R_+$ for $t<0$), and real analytic on the set where it is positive; cf.\ \cite{Biane1997b} for the $t>0$ case, and \cite{Zhong2013} for the $t<0$ case.  Section \ref{section free prob} has further discussion of the measures $\nu_t$ and their relevance to free probability theory.
\end{thmdef}

For $t>0$, $\nu_t$ was identified as $\lim_{N\to\infty}\E(\widetilde{\nu}^N_t)$ in \cite{Biane1997c}, and independently 
in \cite{Rains1997}.  In the latter case, the convergence was proved to be weakly almost sure for polynomial test functions.  Our first main theorem weakens the regularity conditions requires for the almost sure convergence.

\begin{theorem} \label{thm 1} For $t>0$ and $N\in\N$, let $\widetilde{\nu}^N_t$ and $\nu_t$ be the measures in (\ref{e.defempphi}) and Definition \ref{d.nuNew}. Then $\widetilde{\nu}^N_t$ converges to $\nu_t$ weakly in probability:
\begin{equation} \label{e.conv1} \P\left(\left|\int_\T f\,d\widetilde{\nu}^N_t - \int_\T f\,d\nu_t\right|>\e\right)\to 0, \qquad \e>0,\quad f\in C(\T). \end{equation}
Moreover, if $1<p<\frac32$ and $f$ is in the Sobolev space $H_p(\U)$ (cf.\ Definition \ref{d.Sobolev}),  then the convergence is almost sure, and 
\begin{equation} \label{e.conv2} \Var\left(\int_\T f\,d\widetilde{\nu}_t^N\right) \le \frac{C(t,p)}{N^{2p-1}}\|f\|_{H_p(\U)}^2 \end{equation}
for some constant $C(t,p)<\infty$ that depends continuously on $t$ and $p$.  Finally, if $f\in H_p(\U)$ with $p\ge\frac32$, then $f$ is Lipschitz on $\U$, and
\begin{equation} \label{e.conv3} \Var\left(\int_\T f\,d\widetilde{\nu}_t^N\right) \le \frac{2t}{N^2}\|f\|_{\Lip(\U)}.\end{equation}
See (\ref{e.LipU1}) for the definition of the Lipschitz norm on $\U$.
\end{theorem}
\noindent By taking $f\in C(\U)$ close to the indicator function of any given arc, (\ref{e.conv1}) and (\ref{e.conv2}) show that the density of eigenvalues of $U^N_t$ converges, in a fairly strong sense, to $\nu_t$.  We prove Theorem \ref{thm 1} (on page \pageref{proof of thm 1}) incorporating some estimates from \cite{Levy2010} with a Fourier cut-off argument.  Note: in \cite{Levy2010}, the (Gaussian) fluctuations of the empirical integrals $\int_\U f\,d\widetilde{\nu}^N_t$ are computed: they are on the scale of the Sobolev space $H_{1/2}(\U)$ as $t\to\infty$.  We conjecture that the $O(1/N^{2p-1})$ in (\ref{e.conv2}) can be improved to $O(1/N^2)$, and that therefore the a.s.\ convergence holds for $f\in H^p(\U)$ for any $p>\frac12$.  At the end of Section \ref{section proof of thm 2}, we discuss how tighter bounds on the constants from Section \ref{section C(s,t,P)} would lead to this minimal-regularity conjecture.


%
As most matrices in $\GL_N$ are not normal, there are limits to what we can say about the empirical measure $\widetilde{\phi}^N_{s,t}$.  The following is a natural analogue of Theorem \ref{thm 1} in this context.

\begin{theorem} \label{thm 1.1} For $s,t>0$ with $s>t/2$ and $N\in\N$, the empirical eigenvalue measure $\widetilde{\phi}_{s,t}^N$ of (\ref{e.defempphi}) converges ultra-analytically almost surely to $\nu_{s-t}$.  That is: if $f(z) = \sum_n a_n z^n$ is in the ultra-analytic Gevrey class $G_\sigma(\C^\ast)$ (meaning $\|f\|_{G_\sigma}^2 \equiv \sum_n |a_n|^2 e^{2\sigma n^2} <\infty$; cf.\ Definition \ref{d.Gevrey}) for some $\sigma>s$, then
\begin{align} \label{e.holconv1} &\left|\E\left(\int_{\C^\ast} f\,d\widetilde{\phi}_{s,t}^N\right) - \int_{\C^\ast} f\,d\nu_{s-t}\right| \le \frac{C_1(s)}{N^2}\|f\|_{G_\sigma}, \quad \text{and} \\
 \label{e.holconv2} &\qquad\quad \Var\left(\int_{\C^\ast} f\,d\widetilde{\phi}_{s,t}^N\right) \le \frac{C_2(s)}{N^2}\|f\|_{G_\sigma}^2, \end{align}
for some constants $C_1(s),C_2(s)<\infty$ that depend continuously on $s$ (and are independent of $t$).
\end{theorem}
\noindent To be clear, the class $G_\sigma(\C^\ast)$ of test functions is not rich enough to approximate indicator functions of disks, and so Theorem \ref{thm 1.1} does not necessarily imply that the density of eigenvalues converges to $\nu_{s-t}$.  The proof of Theorem \ref{thm 1.2} is on page \pageref{proof of thm 1.1}.

We also consider the convergence of the density of {\bf singular values} of $Z^N_{s,t}$; i.e.\ the square roots of the eigenvalues of the positive-definite matrix $Z^N_{s,t}(Z^N_{s,t})^\ast$.

\begin{definition} \label{d.eta} Let $\M_N^{>0}$ denote the set of positive definite $N\times N$ matrices.  The map $\Phi\colon \GL_N\to \M_N^{>0}$ given by $\Phi(Z) = ZZ^\ast$ is a smooth surjection.  Let $\widetilde{\eta}^N_{s,t}$ be the empirical eigenvalue measure of $\Phi(Z_{s,t}^N)$.
\end{definition}

\begin{theorem} \label{thm 1.2} For $s,t>0$ with $s>t/2$ and $N\in\N$, the empirical eigenvalue measure $\widetilde{\eta}^N_{s,t}$ of Definition \ref{d.eta} converges ultra-analytically almost surely to $\nu_{-t}$:  if $f$ is in the Gevrey class $G_\sigma(\C^\ast)$ for some $\sigma>4s$, then
\begin{align} \label{e.posconv1} &\left|\E\left(\int_0^\infty f\,d\widetilde{\eta}_{s,t}^N\right) - \int_0^\infty f\,d\nu_{-t}\right| \le \frac{C_1(4s)}{N^2}\|f\|_{G_\sigma}, \qquad \text{and} \\
\label{e.posconv2} &\qquad\quad \Var\left(\int_0^\infty f\,d\widetilde{\eta}_{s,t}^N\right) \le \frac{C_2(4s)}{N^2}\|f\|_{G_\sigma}^2, \end{align}
where the constants $C_1(\cdot)$ and $C_2(\cdot)$ are the same ones given in Theorem \ref{thm 1.1}.
\end{theorem}

\noindent The proof of Theorem \ref{thm 1.2} is on page \pageref{proof of thm 1.2}.  It is likely that (\ref{e.posconv2}) holds for much less regular test functions, as in Theorem \ref{thm 1}.  Equation (\ref{e.posconv1}), in the special case of polynomial test functions, was stated without proof at the end of \cite{Biane1997c}, where it was alluded that it follows from combinatorial representation-theoretic tools like used earlier in that paper.  Our present  approach is more geometric.  In fact, we give a unified approach to Theorems \ref{thm 1}, \ref{thm 1.1} and \ref{thm 1.2}, which applies to the much more general context of the noncommutative distribution of $Z^N_{s,t}$; cf.\ Section \ref{section NC distributions}.

\begin{theorem} \label{thm 2} Let $s,t>0$ with $s>t/2$, and let $\widetilde{\p}^N_{s,t}$ denote the empirical noncommutative distribution of $Z^N_{s,t}$; cf.\ Definition \ref{d.emp3}.  There exists a noncommutative distribution $\p_{s,t}$ (cf.\ Definition \ref{d.ncdist0}) such that $\widetilde{\p}^N_{s,t}\to \p_{s,t}$ weakly almost surely: for each noncommutative Laurent polynomial $f\in\C\langle A,A^{-1},A^\ast,A^{-\ast}\rangle$,
\begin{align} \label{e.ncdist1} &\left|\E[\widetilde{\p}^N_{s,t}(f)] - \p_{s,t}(f)\right| \le \frac{C_1(s,t,f)}{N^2}, \qquad \text{and}  \\
\label{e.ncdist2} &\quad \Var[\widetilde{\p}^N_{s,t}(f)] \le \frac{C_2(s,t,f)}{N^2},
\end{align}
for some constants $C_1(s,t,f),C_2(s,t,f)<\infty$ that depend continuously on $s,t$.
\end{theorem}

Let $\tr(Z) = \frac1N\Tr(Z)$ denote the normalized trace on $\M_N$.  Theorem \ref{thm 2} asserts that all of the random trace moments $\tr((Z^N_{s,t})^{\ex_1}\cdots (Z^N_{s,t})^{\ex_n})$ (for $\ex_1,\ldots,\ex_n \in \{\pm 1,\pm\ast\}$) converge almost surely to their means.  In fact, our techniques show the stronger claim that all {\em products} of such trace moments also have $O(1/N^2)$-variance, hence also describing the fluctuations of these random variables. The proof of Theorem \ref{thm 2} is on page \pageref{proof of thm 2}. 


\begin{remark} \label{remark thm2-->} Restricting all test functions to (Laurent) polynomials, Theorem \ref{thm 1} is the special case $(s,t)\mapsto (t,0)$ of Theorem \ref{thm 2}; and Theorems \ref{thm 1.1} and \ref{thm 1.2} are achieved by taking $f$ to depend only on $Z$ in the first case, and only on $ZZ^\ast$ in the second. \end{remark}

The essential idea behind the above concentration results can be described succinctly in the unitary case as follows.  Since the solution $h(t,\cdot)$ to the heat equation $\del_t h = \frac12\Delta_{\U_N}h$ with initial condition $h(0,U) = f(U)$ is given by convolution against the heat kernel (cf.\ \cite{Hall2001a}),
\begin{equation}\label{e.conv0} h(t,U) = \int_{\U_N} f(UV)\rho^N_t(dV), \end{equation}
evaluating this convolution at the identity shows $h(t,I_N)$ is the integral of $f$ against the heat kernel $\rho^N_t$.  But $h(t,\cdot)$ may also be represented in terms of the {\em heat semigroup}, $h(t,\cdot) = e^{\frac{t}{2}\Delta_{\U_N}}f$; thus we have
\begin{equation} \label{e.hk0} \int_{\U_N} f\,d\rho^N_t = \left(e^{\frac{t}{2}\Delta_{\U_N}}f\right)(I_N). \end{equation}
In fact, (\ref{e.hk0}) determines the measure $\rho^N_t$ when taken over all $f\in C(\U_N)$; we take it as the definition of $\rho^N_t$ in (\ref{e.hkm1}) below.  Now, as explained below in Section \ref{section intertwine} following \cite[Theorem 1.18]{DHK2013}, on a sufficiently rich space of functions, $\Delta_{\U_N}$ has a decomposition
\begin{equation} \label{e.int0} \Delta_{\U_N} = D_N + \frac{1}{N^2}L_N \end{equation}
where $D_N$ and $L_N$ are first- and second-order differential operators, both uniformly bounded in $N$; they are given explicitly as intertwining operators in Theorem \ref{thm intertwines}.  In fact, $D_N$ has a limit as $N\to\infty$, which we can think of as the {\em generator} of free unitary Brownian motion; cf.\ \cite{Biane1997c} and Section \ref{section free prob}.  Hence, in the limit as $N\to\infty$, the heat operator $e^{\frac{t}{2}\Delta_{\U_N}}$ behaves as the flow of a vector field; i.e.\ it is an algebra homomorphism, which shows that variances vanish in the limit.  The same idea holds in the $\GL_N$-case as well, in the much larger context of the ``test-functions''  (noncommutative polynomials) of noncommutative distributions; cf.\ Definition \ref{d.convdist}.

These same ideas allow us to prove a stronger form of convergence of these empirical distributions.
\begin{theorem} \label{thm 3} Fix $s,t>0$ with $s>t/2$.  Let $(\A,\t)$ be a noncommutative probability space (Definition \ref{d.ncps}) that contains the almost sure weak limits $u_t$ and $z_{s,t}$ of $U^N_t$ and $Z^N_{s,t}$; cf.\ Theorem \ref{thm 2}.  Then, for any noncommutative polynomial $f\in\C\langle A,A^\ast\rangle$, and any even integer $p\ge 2$,
\begin{align*} &\|f(U^N_t,(U^N_t)^\ast)\|_{L^p(\M_N,\tr)} \to \|f(u_t,u_t^\ast)\|_{L^p(\A,\t)}\; a.s. \quad \text{as} \;\; N\to\infty, \quad \text{and} \\
&\|f(Z^N_{s,t},(Z^N_{s,t})^\ast)\|_{L^p(\M_N,\tr)} \to \|f(z_{s,t},z_{s,t}^\ast)\|_{L^p(\A,\t)}\; a.s. \quad \text{as} \;\; N\to\infty.
\end{align*}
\end{theorem}
\noindent Section \ref{section strong convergence} is devoted to Theorem \ref{thm 3}, where the noncommutative $L^p$-norms are defined and discussed.

\subsection{Discussion}

The problems discussed above are natural extensions of now well-known theorems in random matrix theory.  Let us be slightly more general for the moment.  Let $\rho^N$ be a probability measure on $\M_N$, and let $A_N$ be a random matrix with $\rho^N$ as its joint law of entries.  Denote
\begin{equation} \label{e.emp1} \widetilde{\nu}^N = \frac1N\sum_{\lambda\in\Lambda(A_N)}\delta_{\lambda} \end{equation}
the empirical eigenvalue measure of $A_N$.  If the support of $\rho^N$ is contained in the normal matrices $\M^\nor_N$, then {\em empirical integrals} against measurable test functions $f\colon\C\to\C$ can be computed by
\begin{equation} \label{e.emp1.5} \int_{\C} f\,d\widetilde{\nu}^N = \tr\circ f_N, \end{equation}
where the function $f_N\colon\M_N^\nor\to\M_N^\nor$ is given by {\em measurable functional calculus}; cf.\ Section \ref{section functional calculus} below.   In particular, (\ref{e.emp1.5}) will often be used to compute expectations against continuous functions:
\begin{equation} \label{e.emp2} \E\left(\int_{\C} f\,d\widetilde{\nu}^N\right) = \int_{\M_N^\nor} (\tr\circ f_N)\,d\rho^N, \qquad f\in C_c(\C). \end{equation}
The most well-known example of such a normal (in fact Hermitian) empirical eigenvalue measure comes from Wigner's semicircle law; cf.\ \cite{Wigner1955,Wigner1957,Wigner1958}.  In the original Gaussian case, $\rho^N$ is supported on Hermitian matrices, with
\begin{equation} \label{e.expot1} \rho^N(dX) = c_N e^{-N\Tr(X^2)}\,dX \end{equation}
where $dX$ denotes the Lebesgue measure on Hermitian matrices (coordinatized by the real and imaginary parts of the upper-triangular entries), and $c_N$ is a normalization constant.  This measure is known as the $\mathrm{GUE}_N$  or {\em Gaussian Unitary Ensemble};
it is equivalently described by insisting that the upper-triangular entries of the Hermitian random matrix $X$ are i.i.d.\ normal random variables of variance $1/N$.  Wigner proved that, in this case, the empirical eigenvalue measure converges weakly in expectation to the {\em semicircle law} $\varsigma(dx) = \frac{1}{2\pi}\sqrt{(4-x^2)_+}\,dx$.  That is to say: Wigner proved that the quantities in (\ref{e.emp2}) converge to the relevant integrals against $d\varsigma$.  It was shown later \cite{Arnold1967,Arnold1971,Bai1999} that this convergence is {\em weakly almost sure}, in the sense that the random variables $\int f\,d\nu^N$ converge to their means almost surely.

\begin{remark} Having realized all requisite random matrices (of all sizes $N\in\N$) over a single probability space $(\Omega,\mathscr{F},\P)$, proving almost sure convergence amounts to showing that the variances tend to $0$ summably-fast (by Chebyshev's inequality and the Borel-Cantelli lemma). \end{remark}

Much of the modern theory of random matrices is concerned with generalizations of Wigner's example in one of two ways: either to other measures $\rho^N$ on Hermitian matrices that make the upper-triangular entries i.i.d., or or to measures with densities generalizing the form of (\ref{e.expot1}), for example by replacing $\Tr(X^2)$ with a different (sufficiently convex) potential.  A great deal is understood in both these arenas about the empirical measures and many other statistics of the random eigenvalues; the interested reader should consult \cite{AGZBook}.

Another well-studied example is the Haar measure $\rho^N = \mathrm{Haar}(\U_N)$ on the  the unitary group $\U_N$. Unitary matrices are normal, and so (\ref{e.emp1.5}) characterizes the empirical eigenvalue measures; in this case, they are known (cf.\ \cite{Diaconis2001}) to converge weakly almost surely to the uniform probability measure on $\U$.  In both this case and the Wigner ensembles described above, stronger convergence results are known, such as in Theorem \ref{thm 3} above.

\begin{remark}  If, instead of $\U_N$, we take the additive Lie group of Hermitian matrices, the heat kernel is precisely the Gaussian measure (\ref{e.expot1}), where $N$ is replaced by $N/t$ on the right-hand-side.  The space of Hermitian matrices can be identified as $i\u_N$, where $\u_N = \{X\in\M_N\colon X^\ast=-X\}$ is the Lie algebra of $\U_N$; thus, the $\mathrm{GUE}_N$ is the Lie algebra version of the heat kernel on $\U_N$.  As $t\to\infty$, the heat kernel measure $\rho_t^N$ on $\U_N$ converges to the Haar measure.  In this sense, the heat kernel measures considered in the present paper fit into a larger scheme of well-studied random matrix ensembles. \end{remark}

The support of the heat kernel measures $\mu^N_{s,t}$ on $\GL_N$ consists largely of non-normal matrices, and so measurable functional calculus is not available.  It is for this reason that our analysis is restricted to holomorphic test functions in this case.  Nevertheless, the results presented in Theorems \ref{thm 1.1} -- \ref{thm 2} are new; in particular, the existence of the noncommutative distribution $\p_{s,t}$ in Theorem \ref{thm 2} was part of a conjecture posed in \cite{Biane1997c}.  The full conjecture deals with the limit of the stochastic process $t\mapsto Z^N_{t,t}$, the Brownian motion on $\GL_N$ which, for each fixed $t$, has distribution $\mu^N_{t/2}$.  In the present paper, we deal only with a single $t>0$, with all theorems proved with bounds that are uniform for $t$ in compact intervals.

\section{Background}

In this section, we give concise discussions of the necessary constructs for this paper: heat kernel analysis on the groups $\U_N$ and $\GL_N$; regularity of test functions (Sobolev spaces and Gevrey classes); measurable functional calculus on $\U_N$ and holomorphic functional calculus on $\GL_N$; and noncommutative probability theory (in particular free probability and free multiplicative convolution).  For general reference, readers are directed to the monograph \cite{Robinson1991} for heat kernel analysis on Lie groups, and the lecture notes \cite{NicaSpeicherBook} for a thorough treatment of noncommutative and free probability.

\subsection{Heat Kernels on $\U_N$ and $\GL_N$\label{section heat kernels}} Let $G\subset \M_N$ be a matrix Lie group, with Lie algebra $\mathrm{Lie}(G)$; relevant to this paper are $\U_N$ with $\mathrm{Lie}(\U_N) = \u_N = \{X\in \M_N\colon X^\ast = -X\}$, and $\GL_N$ with $\mathrm{Lie}(\GL_N) = \gl_N = \M_N$.  Note that $\gl_N = \u_N\oplus i\u_N$.  Hence, if $\beta_N$ is a basis for $\u_N$ as a real vector space, then $\beta_N$ is also a basis for $\gl_N$ as a complex vector space.

We will use the following (scaled) Hilbert-Schmidt inner product on $\gl_N$:
\begin{equation} \label{e.innprod1} \langle \xi,\zeta\rangle_N \equiv N\Tr(\xi\zeta^\ast) = N^2\tr(\xi\zeta^\ast), \qquad \xi,\zeta\in\gl_N. \end{equation}
Restricted to $\u_N$, this inner product is $\mathrm{Ad}_{\U_N}$-invariant, and real valued:
\begin{equation} \label{e.innprod2} \langle X,Y\rangle_N = -N\Tr(XY), \qquad X,Y\in\u_N. \end{equation}
The scaling chosen here is consistent with the scaling in (\ref{e.expot1}); as we will see in the following, it is the unique scaling that leads to limit distributions as $N\to\infty$.

\begin{definition} \label{d.deriv1} Let $G$ be a Lie group and $\xi\in\mathrm{Lie}(G)$.  Then the exponential $e^{t\xi}$ is in $G$ for $t\in\R$.  The {\bf left-invariant vector field} or {\bf derivative} associated to $\xi$ is the operator $\del_\xi$ on $C^\infty(G)$ defined by
\begin{equation} \label{e.deriv1} (\del_\xi f)(g) = \left.\frac{d}{dt}\right|_{t=0} f(ge^{t\xi}). \end{equation}
\end{definition}

\begin{definition} \label{d.Laplace} Let $\beta_N$ be an orthonormal basis (with respect to (\ref{e.innprod2})) for $\u_N$.  The {\bf Laplace} operator on $C^\infty(\U_N)$ is 
\begin{equation} \label{e.Delta_U} \Delta_{\U_N} = \sum_{X\in\beta_N} \del_X^2. \end{equation}
The {\bf Laplace} operator on $C^\infty(\GL_N)$ is
\begin{equation} \label{e.Delta_GL} \Delta_{\GL_N} = \sum_{X\in\beta_N} \left(\del_X^2 + \del_{iX}^2\right). \end{equation}
More generally, for $s,t\in\R$, define the operators $A^N_{s,t}$ on $C^\infty(\GL_N)$ by
\begin{equation} \label{e.Ast} A^N_{s,t} = \left(s-\frac{t}{2}\right)\sum_{X\in\beta_N} \del_X^2 + \frac{t}{2}\sum_{X\in\beta_N} \del_{iX}^2. \end{equation}
A routine calculation shows that these definitions do not depend on the particular orthonormal basis used.
\end{definition}

\begin{remark} \label{r.Delta} \begin{itemize}
\item[(1)] The operator $\Delta_{\U_N}$ is the Casimir element in the universal enveloping algebra $\mathcal{U}(\u_N)$.  Since the inner product (\ref{e.innprod2}) is $\mathrm{Ad}$-invariant, $\Delta_{\U_N}$ commutes with the left- and right-actions of $\U_N$ on $C^\infty(\U_N)$; i.e.\ it is bi-invariant.  It is equal to the Laplace-Beltrami operator on $\U_N$ associated to the bi-invariant Riemannian metric induced by (\ref{e.innprod2}).
\item[(2)] The non-semisimple Lie group $\GL_N$ possesses no $\mathrm{Ad}$-invariant inner product.  Eq.\ (\ref{e.Delta_GL}) matches the Laplace-Beltrami operator on $\GL_N$ associated to the left-invariant Riemannian metric induced by (\ref{e.innprod1}).
\item[(3)] The interpolating operator $A_{s,t}^N$ is negative-definite when $s,t>0$ and $s>t/2$; in this regime, it is essentially self-adjoint on $L^2(\GL_N)$ equipped with any right Haar measure; cf.\ \cite{Driver1999,Hall2001b}.  In the special case $s=t$, $A^N_{t,t} = \frac{t}{2}\Delta_{\GL_N}$.  Note also that $t\Delta_{\U_N} = \left.A^N_{t,0}\right|_{C^\infty(\U_N)}$.
\end{itemize}
\end{remark}

\begin{definition} \label{d.hkm} For $t>0$, the {\bf heat kernel measure} $\rho^N_t$ on $\U_N$ is the unique probability measure which satisfies
\begin{equation} \label{e.hkm1} \E_{\rho^N_t}(f) \equiv \int_{\U_N} f\,d\rho^N_t = \left(e^{\frac{t}{2}\Delta_{\U_N}}f\right)(I_N), \qquad f\in C(\U_N). \end{equation}
Additionally, for $s>t/2$, the {\bf heat kernel measure} $\mu^N_{s,t}$ on $\GL_N$ is the unique probability measure which satisfies
\begin{equation} \label{e.hkm2} \E_{\mu^N_{s,t}}(f) \equiv \int_{\GL_N} f\,d\mu^N_{s,t} = \left(e^{\frac12A^N_{s,t}}f\right)(I_N), \qquad f\in C_c(\GL_N). \end{equation}
In particular, the standard heat kernel measure on $\GL_N$ is $\mu^N_{t/2} = \mu^N_{t,t}$; cf.\ Remark \ref{r.Delta}(3).
\end{definition}

\begin{remark} \label{remark must}\begin{itemize}
\item[(1)] The operators $e^{\frac{t}{2}\Delta_{\U_N}}$ and $e^{\frac12A_{s,t}}$ can be made sense of with PDE methods (since $\Delta_{\U_N}$ and $A_{s,t}^N$ are elliptic) or functional analytic methods (since they are essentially self-adjoint).  In most of our applications, the test functions $f$ will be polynomials in the entries of the matrix argument, and the operators can interpreted via the power series expansion of $\exp$.
\item[(2)] Eq.\ (\ref{e.hkm2}) holds, a priori, only for compactly-supported continuous test functions.  In fact, it holds much more generally; in particular, it holds for any function $f$ that is polynomial in the matrix entries.  This follows from Langland's Theorem \cite[Theorem 2.1 (p.\ 152)]{Robinson1991}; see also \cite[Appendix A]{DHK2013}.
\item[(3)] More generally, for $s,t>0$ and $s>t/2$, there is a strictly-positive smooth heat kernel function
\[ h^N_{s,t}\colon \GL_N\times\GL_N\to\R_+\]
such that, for $f\colon\GL_N\to\C$ of sufficiently slow growth (as in (2) above),
\[ \left(e^{\frac12 A^N_{s,t}}f\right)(Z) = \int h^N_{s,t}(Z,W)f(W)dW \]
where $dW$ denotes the right-Haar measure on $\GL_N$.  Thus, the density of $\mu^N_{s,t}$ is thus $h^N_{s,t}(I_N,\cdot)$; cf.\ \cite{Driver1999,Hall2001b}.  Since $h^N_{s,t}$ is real-valued, for any $f$ in the domain of $e^{\frac12A_{s,t}^N}$, it follows that
\[ \overline{e^{\frac12A_{s,t}^N}\overline{f}} = e^{\frac12A_{s,t}^N}f, \]
where $\overline{f}(Z) = \overline{f(Z)}$ is the complex conjugate.  Setting $t=0$ shows that the same property holds for the heat operator $e^{\frac{s}{2}\Delta_{\U_N}}$.  This will be useful in the proof of Lemma \ref{l.[C,A]} below.
\end{itemize}
\end{remark}

\begin{remark}
Had we taken the usual (unscaled) Hilbert-Schmidt inner product $(X,Y) = -\Tr(XY)$ in (\ref{e.innprod2}), the resulting heat kernel measure on $\U_N$ would have been $\rho^N_{Nt}$.  This is the approach taken in \cite{Levy2008,Levy2010}, and instead the heat kernel is evaluated at time $t/N$ to compensate.  In that sense, our limiting concentration results can be interpreted as statements about the heat kernel in a neighborhood of $t=0$.
\end{remark}

%
%

\subsection{The Heat Kernel on $\U$, Sobolev Spaces, and Gevrey Classes \label{section Gevrey}}

If $f\in L^2(\U)$, its Fourier expansion is given by
\[ f = \sum_{n\in\Z} \hat{f}(n)\chi_n, \qquad \hat{f}(n) = \langle f,\chi_n\rangle_{L^2(\U)} = \int_\U f(u) u^{-n}\,du, \]
where $\chi_n(u) = u^n$ for $u\in\U$ and $n\in\Z$, and $du$ denotes the normalized Haar measure on $\U$.
\begin{definition} \label{d.Sobolev} For $p>0$, the {\bf Sobolev space} $H_p(\U)$ is defined by
\begin{equation} \label{e.d.Hp} H_p(\U) = \left\{f\in L^2(\U)\colon \|f\|_{H_p}^2\equiv \sum_{n\in\Z} (1+n^2)^p|\hat{f}(n)|^2 < \infty\right\}. \end{equation}
\end{definition}
\noindent Note that $H_0(\U) = L^2(\U)$.  The definition makes sense even for $p<0$, where the elements are no longer $L^2$-functions but rather distributions.  If $k\ge1$ is an integer, and $p>k+\frac12$, then $C^{k-1}\subset H_p(\U)\subset C^k(\U)$; it follows that $H_\infty(\U)\equiv\bigcap_{p\ge 0} H_p(\U) = C^\infty(\U)$.  For $\frac12<p\le\frac32$, functions in $H_p(\U)$ are H\"older continuous of any modulus $<p-\frac12$, but generically not smoother.  For $p\le\frac12$, $H_p(\U)$ functions are generally not continuous.  These are standard Sobolev imbedding theorems (that hold for smooth manifolds); for reference, see \cite[Chapter 5.6]{Evans2010} and \cite[Chapter 3.2]{Saloff-Coste2002}.

It is elementary to describe the heat semigroup on $\U=\U_1$ in terms of Fourier expansions.  Indeed,
\begin{equation} \label{e.Delta1} (\Delta_{\U_1}f)(u) = -\frac{\del^2}{\del u^2}f(u) \end{equation}
(Here $u=e^{i\theta}$; (\ref{e.Delta1}) is more commonly written as $\left(\Delta_{\U_1}f\right)(e^{i\theta}) = \frac{\del^2}{\del\theta^2}f(e^{i\theta})$ in PDE textbooks.)  Hence, the characters $\chi_n$ are eigenfunctions $\Delta_{\U_1} \chi_n = -n^2\chi_n$, and so
\begin{equation} \label{e.HU1chin} \Delta_{\U_1}\chi_n = e^{-\frac{t}{2}n^2}\chi_n, \qquad n\in\Z,\; t\in\R. \end{equation}
It follows that the heat semigroup is completely described on $L^2(\T)$ as a Fourier multiplier
\begin{equation} \label{e.Fourier1} e^{\frac{t}{2}\Delta_{\U_1}}f = \sum_{n\in\Z} e^{-\frac{t}{2}n^2}\hat{f}(n)\chi_n. \end{equation}

Let $f\in L^2(\T)$, and for $t>0$ let $f_t = e^{\frac{t}{2}\Delta_{\U_1}}f$.  Then (\ref{e.Fourier1}) shows that $\hat{f}_t(n) = e^{-\frac{t}{2}n^2}\hat{f}(n)$.  In particular, this means that
\begin{equation} \label{e.Fourier2} \sum_{n\in\Z} e^{tn^2}|\hat{f}_t(n)|^2 = \sum_{n\in\Z} |\hat{f}(n)|^2 = \|f\|_{L^2(\T)} <\infty. \end{equation}
It follows that $f_t\in H_\infty(\U)=C^\infty(\U)$.  It is, in fact, {\em ultra-analytic}.

\begin{definition} \label{d.Gevrey} Let $\sigma>0$.  The {\bf Gevrey class} $G_\sigma(\T)$ consists of those $f\in L^2(\T)$ such that
\begin{equation} \label{e.Gevrey} \|f\|_{G_\sigma}^2 \equiv \sum_{n\in\Z} e^{2\sigma n^2}|\hat{f}(n)|^2 <\infty. \end{equation}
More generally, the Gevrey class $G^{s,p}_\sigma(\T)$ consists of those $f\in L^2(\T)$ for which
\[ \|f\|_{G^{s,p}_\sigma}^2 \equiv \sum_{n\in\Z} (1+n^2)^p e^{2\sigma |n|^{1/s}} |\hat{f}(n)|^2 <\infty, \]
so that $G_\sigma(\T)$ is the $s=1/2,\ p=0$ case of $G^{s,p}_\sigma(\T)$.
\end{definition}
\noindent These spaces arise naturally in the analysis of some non-linear parabolic PDEs, cf.\ \cite{Ferrari1998,Foias1989,Levermore1997}.  The superexponent $s$ is usually taken to be $1$, in which case $G_\sigma^{1,p}$ is a Hilbert space of real analytic functions.  For $s>1$, Gevrey functions in $G^{s,p}_\sigma$ are $C^\infty$ but generally not analytic, and when $s=\infty$ we recover the Sobolev spaces; thus the two-parameter family $G^{s,p}_\sigma$ interpolates between $C^\infty$ functions and analytic functions for $s\ge 1$.

In the regime $0<s<1$ such functions are called {\bf ultra-analytic}.  Indeed, if  if $f\in G_\sigma(\U)$ for some $\sigma>0$, then $f$ has a unique analytic continuation to a holomorphic function on $\C^\ast$ given by the convergent Laurent series $f(z) = \sum_{n=-\infty}^\infty \hat{f}(n) z^n$.  (The holomorphic $n\ge 0$ sum converges uniformly on $\C$ and the principal part $n<0$ converges uniformly on $\C^\ast$ due to the fast decay of the coefficients.)  We therefore refer to the set of such holomorphic functions as
\begin{equation} \label{e.GC*} G_\sigma(\C^\ast) = \left\{f\in\mathrm{Hol}(\C^\ast)\colon \left.f\right|_{\U}\in G_\sigma(\U)\right\} = \left\{f(z) = \sum_{n\in\Z} a_nz^n \colon \|f\|_{G_\sigma}^2 \equiv \sum_{n\in\Z} e^{2\sigma n^2} |a_n|^2 < \infty\right\}. \end{equation}
Note, as shown in (\ref{e.Fourier2}), the Gevrey class $G_\sigma$ characterizes the domain of the backwards heat flow:
\begin{equation} \label{e.GevreyHeat} G_\sigma(\T) = \left\{f\in L^2(\U)\colon e^{-\frac{t}{2}\Delta_{\U_1}}f\;\text{ exists in }L^2(\T)\text{ for small time }\;0\le t \le 2\sigma\right\}. \end{equation}

\subsection{Functional Calculus and Empirical Measures\label{section functional calculus}}

For a normal matrix $X\in\M_N^\nor$, the spectral theorem asserts that there are mutually orthogonal projection operators $\{\Pi^X_\lambda\colon \lambda\in\Lambda(X)\}\subset\mathrm{End}(\C^N)$ so that
\[ X = \sum_{\lambda\in\Lambda(X)} \lambda \Pi^X_\lambda. \]
For any measurable function $f\colon\C\to\C$, define $f_N\colon \M_N^\nor\to\M_N^\nor$ by
\begin{equation} \label{e.fc1} f_N(X) = \sum_{\lambda\in\Lambda(X)} f(\lambda)\Pi^X_\lambda. \end{equation}
That is: if $X = U\Lambda U^\ast$ is any unitary diagonalization of $X$, then $f_N(X) = Uf(\Lambda) U^\ast$ where $[f(\Lambda)]_{jj} = f([\Lambda]_{jj})$ for $1\le j\le N$.  The map $f\mapsto f_N$ is called {\bf measurable functional calculus}.  We adhere to the notation we used in \cite{DHK2013}; in \cite{Biane1997b}, $f_N$ was denoted $\theta^N_f$. 

Let $\rho^N$ be a probability measure supported in $\M_N^\nor$. The linear functional
\[ C_c(\C)\ni f\mapsto \int_{\M_N^\nor} \tr(f_N(X))\,\rho^N(dX) \]
is easily verified to be positive; also, if $\overline{\mathbb{D}}_r$ is the disk of radius $r>0$, then
\[ \int_{\M_N^\nor} \tr\left([\1_{\overline{\mathbb{D}}_r}]_N(X)\right)\,\rho^N(dX) \to 1 \quad \text{as} \quad n\to\infty. \]
Hence,  by the Riesz Representation Theorem \cite[Theorem 2.14]{RudinBook}, there is a unique Borel probability measure $\nu^N$ on $\C$ such that
\begin{equation} \label{e.eigmeas} \int_{\C} f\,d\nu^N = \int_{\M_N^\nor} (\tr\circ f_N)\,d\rho^N, \qquad f\in C_c(\C). \end{equation}
Comparing to (\ref{e.emp2}), this Riesz measure $\nu^N$ is the mean of the empirical measure $\widetilde{\nu}^N$ (\ref{e.emp1}).  In particular, if $\nu$ is a (deterministic) measure such that $\widetilde{\nu}^N\rightharpoonup \nu$ weakly in probability, then we must have $\nu^N\rightharpoonup \nu$ weakly.  

\begin{remark} \label{r.cpct} In the special case that $\supp(\rho^N)$ is compact, the Weierstrass approximation theorem shows that (\ref{e.eigmeas}) is equivalent to equating the moments of $\nu^N$ with the {\em trace moments} of $\rho^N$:
\begin{equation} \label{e.moments1} \int_\C x^n\bar{x}^m\,\nu^N(dx) = \int_{\M_N^\nor} \tr(X^n(X^\ast)^m)\,\rho^N(dX). \end{equation}
In our first case of interest where $\rho^N_t$ is the heat kernel on the compact group $\U_N$, this amounts to defining $\nu^N_t$ by its integrals against {\em Laurent polynomials}; cf.\ Section \ref{section NC distributions}.
\end{remark}


If $\supp(\rho^N)$ is not contained in $\M_N^\nor$, measurable functional calculus is not available.  Instead, we can consider {\em holomorphic} test functions.  In the case of interest (the heat kernel $\mu^N_{s,t}$ on $\GL_N$), all empirical eigenvalues are in $\C^\ast$, so we take $f\in\mathrm{Hol}(\C^\ast)$; for simplicity, we assume the Laurent series $f(z) = \sum_{n=-\infty}^\infty a_n z^n$ converges on all of $\C^\ast$.  (This is not necessary, but it simplifies matters and suffices for our purposes.)  Then the series
\begin{equation} \label{e.Laurent0} f_N(Z) \equiv \sum_{n=-\infty}^\infty a_nZ^n, \end{equation}
where we interpret the $n=0$ term as $a_0I_N$, converges for any $Z\in\GL_N$. The map $f\mapsto f_N$ is called {\bf holomorphic functional calculus}.  We use the same notation as for functional calculus, and this is consistent: if $Z$ is normal and $f$ is holomorphic as above, then the Laurent series (\ref{e.Laurent0}) coincides with the functional calculus map of (\ref{e.fc1}).

Since there are no non-constant positive holomorphic functions, no integration formula like (\ref{e.eigmeas}) can be used to define an ``expected empirical eigenvalue measure'' in this case.  There may or may not exist such a measure $\nu^N$ on $\C$; if it does exist, it will not be uniquely determined by (\ref{e.eigmeas}).  In general, there is just too much information in the trace (noncommutative) moments of a non-normally supported measure $\rho^N$ to be captured by a single measure on $\C$.  Instead, we need the notion of a {\em noncommutative distribution}.
\subsection{Noncommutative Distributions\label{section NC distributions}}

\begin{definition} \label{d.ncps} Let $\A$ be a unital complex $\ast$-algebra.  A {\bf tracial state} $\t\colon\A\to\C$ is a linear functional that is unital ($\t(1) = 1$), tracial ($\t(ab) = \t(ba)$ for $a,b\in\A$), and positive semidefinite ($\t(aa^\ast)\ge 0$ for all $a\in\A$).  If, in addition, $\t(aa^\ast)\ne 0$ for $a\ne 0$, $\t$ is called {\bf faithful}.  The pair $(\A,\t)$ is called a {\bf (faithful, tracial) noncommutative probability space}.  If $\A$ is a $C^\ast$-algebra, we refer to $(\A,\t)$ as a $C^\ast$-probability space; if $\A^\ast$ is a $W^\ast$-algebra (i.e.\ von Neumann algebra), we refer to $(\A,\t)$ as a $W^\ast$-probability space.
\end{definition}
\noindent If $(\Omega,\mathscr{F})$ is a probability space and $\P$ is a probability measure on $(\Omega,\mathscr{F})$, the expectation $\E = \int \cdot\,d\P$ is a faithful tracial state on the algebra $L^\infty(\Omega,\mathscr{F},\P)$ of complex-valued random variables (where $F^\ast = \overline{F}$); thus the {\em probability space} terminology.  Truly noncommutative examples are afforded by $\M_N$ equipped with $\tr$, which is a faithful tracial state.  It is these examples that will be most relevant to us.

In the example $L^\infty(\Omega,\mathscr{F},\P)$, any random variable $F\in L^\infty$ has a probability distribution $\mu_F$ (on $\C$ if the random variables are $\C$-valued), which is the push-forward $\mu_F(B) = \big(F_\ast(\P)\big)(B) = \P(F^{-1}(B))$ for Borel sets $B\subseteq\C$.  In terms of the expectation, this can be written as
\begin{equation} \label{e.ncdist0} \int f\,d\mu_F = \E(f(F)), \qquad f\in C_c(\C). \end{equation}
If $(\A,\t)$ is a noncommutative probability space such that $\A$ is a $W^\ast$-algebra, any measurable $f\colon\C\to\C$ induces (by the spectral theorem) a function $f_\A\colon \A^{\mathrm{nor}}\to\A^{\mathrm{nor}}$ as in (\ref{e.fc1}); here $\A^{\nor}$ refers to the normal operators in $\A$.  The map $f\mapsto f_\A$ is the {\em measurable functional calculus}.  We then define the {\bf distribution} $\mu_a$ of $a\in\A^{\nor}$ to be the unique Borel probability measure on $\C$ mimicking (\ref{e.ncdist0}):
\begin{equation} \label{e.ncdist1'} \int_\C f\,d\mu_a = \t(f_\A(a)), \qquad f\in C_c(\C). \end{equation}
Indeed, (\ref{e.ncdist1'}) determines $\mu_a$ for $f\in C(\sigma(a))$, as the spectrum $\sigma(a)$ is compact (since $a\in\A$ is a bounded operator).  Therefore, as in (\ref{e.moments1}), in (\ref{e.ncdist1'}) we need only use test functions of the form $f(x) = x^n\bar{x}^m$, $n,m\in\N$, so that $f_\A(a) = a^n(a^\ast)^m$.  Hence, in this case, $\mu_a$ is equivalently determined by all moments, through the formula
\begin{equation} \label{e.ncdist2'} \int_\C x^n\bar{x}^m\,\mu_a(dx) = \t(a^n(a^\ast)^m), \qquad n,m\in\N. \end{equation}

\begin{remark} In the special case $(\A,\t) = (\M_N,\tr)$, the distribution of a normal matrix is precisely its empirical eigenvalue measure; cf. (\ref{e.emp1.5}). \end{remark}

If $a$ is a non-normal operator in $(\A,\t)$, it may or may not be the case that there is a measure $\mu_a$ on $\C$ satisfying (\ref{e.ncdist2'}).  Even if there is, these moments do not determine all other moments $\t\big(a^{n_1}(a^{\ast})^{m_1}\cdots a^{n_k}(a^\ast)^{m_k}\big)$.  We therefore {\em define} this collection of moments to be the noncommutative distribution of $a$.  In the spirit of the Riesz theorem identifying measures as linear functionals, this can be formulated as follows.

\begin{definition} \label{d.ncdist0} Let $\C\langle A,A^{-1},A^\ast,A^{-\ast}\rangle$ denote the algebra of {\bf noncommutative Laurent polynomials} in two variables $A$ and $A^\ast$; in other words, $\C\langle A,A^{-1},A^\ast,A^{-\ast}\rangle \cong \C\F_2$ is the complex group algebra of the free group on two generators $A,A^\ast$.  Let $\C\langle A,A^\ast\rangle$ denote the subalgebra of {\bf noncommutative polynomials} in two variables $A,A^\ast$; in other words, $\C\langle A,A^\ast\rangle\cong \C\langle A,A^\ast\rangle$ is the group algebra over the free semigroup $\F_2^+$ generated by $A,A^\ast$.

If $(\A,\t)$ is a noncommutative probability space and $a\in\A$, the {\bf noncommutative distribution} of $a$ is the linear functional $\p_a\colon\C\langle A,A^\ast\rangle\to\C$ defined by
\begin{equation} \label{e.pa1} \p_a(f) = \t\left[f(a,a^\ast)\right], \qquad f\in\C\langle A,A^\ast\rangle \end{equation}
for any element $f = f(A,A^\ast)$.  If $a$ is invertible in $\A$, then $\p_a$ extends uniquely to a linear functional on $\C\langle A,A^{-1},A^\ast,A^{-\ast}\rangle$ by (\ref{e.pa1}).
\end{definition}

\begin{notation} \label{n.EX1} For $n\in\N$, let $\EX_n$ denote the set of all $n$-tuples $\ex\in \{\pm1,\pm\ast\}^n$, and let $\EX_n^+$ be the subset $\{1,\ast\}^n$.  ($\EX_0=\emptyset$ .) For $\ex\in\EX_n$, denote $|\ex|=n$.  Set $\EX = \bigcup_n \EX_n$, and $\EX^+ = \bigcup_n \EX^+_n$. 

Given a $\ast$-algebra $\A$, for $a\in\A$ and $\ex\in\EX^+$, denote $a^\ex = a^{\ex_1}a^{\ex_2}\cdots a^{\ex_n}$ where $n=|\ex|$.  Then $\C\langle A,A^\ast\rangle$ can be described explicitly as
\[ \C\langle A,A^\ast\rangle = \mathrm{span}_\C\left\{A^\ex\colon\ex\in\EX^+\right\}.\]
The vectors $A^\ex$ form a basis for this $\C$-space.  The algebra structure is given by concatenation in $\EX^+$: $A^\ex\cdot A^\d = A^{\ex\d}$ where, if $\ex\in\EX^+_n$ and $\d\in\EX^+_m$, then $\ex\d = (\ex_1,\ldots,\ex_n,\d_1,\ldots,\d_m)\in\EX^+_{n+m}$.

The algebra $\C\langle A,A^{-1},A^\ast,A^{-\ast}\rangle$ is similarly equal to the $\C$-span of $A^\ex$ for $\ex\in\EX$, with product defined by concatenation; but in this case these words are no longer linearly independent (for example $A^\ast AA^{-1} = A^\ast$).  A basis for $\C\langle A,A^{-1},A^\ast,A^{-\ast}\rangle$ consists of {\em reduced words} $A^\ex$ in the sense of free groups.
\end{notation}
\noindent Thus, the noncommutative distribution of $a\in(\A,\t)$ can equivalently be described as the linear functional $\p_a\colon\C\langle A,A^\ast\rangle\to\C$ defined by
\begin{equation} \label{e.pa2} \p_a(A^\ex) = \t(a^\ex), \qquad \ex\in\EX^+. \end{equation}
If $a$ is invertible in $\A$, this extends by the same formula to a linear functional on $\C\langle A,A^{-1},A^\ast,A^{-\ast}\rangle$ (due to the universal property of free groups).

If $a$ is normal, then for any $\ex\in\EX^+$, $a^\ex = a^n(a^\ast)^m$ where $n$ is the number of $1$s and $m$ is the number of $\ast$s in $\ex$.  Hence, in this case, $\p_a$ is completely determined by the measure $\mu_a$ of (\ref{e.ncdist2'}).  Thus $\p_a$ generalizes the classical notion of distribution of a random variable.

We will work largely with the noncommutative probability spaces $(\M_N,\tr)$, often with randomness involved.

\begin{definition} \label{d.emp3} Let $\rho^N$ be a probability measure on $\M_N$, such that all polynomial functions of the matrix entries are in $L^1(\rho^N)$; this condition holds for the heat kernel measures $\mu^N_{s,t}$ on $\GL_N$ by Remark \ref{remark must}(2).  The associated {\bf empirical noncommutative distribution} $\widetilde{\p}^N$ is defined to be the $\mathrm{Hom}(\C\langle A,A^\ast\rangle;\C)$-valued random variable on the probability space $(\M_N,\rho^N)$ given by
\begin{equation} \label{e.emp3} \widetilde{\p}^N(Z)= \p_Z \;\; \text{with respect to the noncommutative probability space} \;\; (M_N,\tr). \end{equation}
That is: $\left(\widetilde{\p}^N(Z)\right)(A^\ex) = \tr(Z^\ex)$ for $\ex\in\EX^+$.  If $\rho^N$ is supported on $\GL_N$, then $\widetilde{\p}^N$ extends to a random linear functional on $\C\langle A,A^{-1},A^\ast,A^{-\ast}\rangle$.  The expectation $\E(\widetilde{\p}^N)$ is defined to be the linear functional on $\C\langle A,A^\ast\rangle$ given by
\begin{equation} \label{e.emp4} \E\left(\widetilde\p^N\right)(f) = \int_{\M_N} \p_Z(f)\,\rho^N(dZ), \qquad f\in\C\langle A,A^\ast\rangle. \end{equation}
\end{definition}
Equations (\ref{e.emp3}) and (\ref{e.emp4}) are natural generalization of (\ref{e.emp1.5}) and (\ref{e.emp2}).  The polynomial-integrability condition we placed on $\rho^N$ guarantees that (\ref{e.emp4}) is a well-defined linear functional; moreover, $\E(\widetilde{\p}^N)$ {\em is} the noncommutative distribution of some random variable.  Indeed, we can construct this random variable in the algebra $\C\langle A,A^\ast\rangle$ itself.  Define the linear functional $\t_{\rho^N}$ on $\C\langle A,A^\ast\rangle$ to verify (\ref{e.emp4}):
\[ \t_{\rho^N}(f) = \int_{\M_N} \p_Z(f)\,\rho^N(dZ) = \int_{\M_N} \tr[f(Z,Z^\ast)]\,\rho^N(dZ). \]
The linear functional $\t_{\rho^N}$ is easily verified to be a tracial state, so $(\A,\t_{\rho^N})$ is a noncommutative probability space; cf.\ Definition \ref{d.ncps}.  It is faithful provided $\supp(\rho^N)$ is infinite.  Let $a\in\C\langle A,A^\ast,\rangle$ denote the coordinate random variable $a(A,A^\ast) = A$; then its noncommutative distribution $\p_a$ with respect to $(\C\langle A,A^\ast\rangle,\t_{\rho^N})$ is, by (\ref{e.pa2}) and (\ref{e.emp4}),
\[ \p_a(A^\ex) = \t_{\rho^N}(a(A)^\ex) = \t_{\rho^N}(A^\ex) = \int_{\M_N} \tr(Z^\ex)\,\rho^N(dZ) = \E\left(\widetilde{\p}^N\right)(A^\ex), \qquad \ex\in\EX^+. \]
Thus, $\E(\widetilde{\p}^N)$ defines a (deterministic) noncommutative distribution which we call the {\bf mean} of $\widetilde{\p}^N$.

\begin{definition} \label{d.convdist} Let $\p^N$ be a sequence of noncommutative distributions; that is, there are noncommutative probability spaces $(\A_N,\t_N)$ with some distinguished elements $a_N\in\A_N$ so that $\p^N = \p_{a_N}$ over $\A_N$.  We say that $\p^N$ {\bf converges weakly} (or {\bf converges in distribution}) if there is a noncommutative distribution $\p$ so that $\p^N(f)\to \p(f)$ for all $P\in\C\langle A,A^\ast\rangle$.  That is: there exists a noncommutative probability space $(\A,\t)$ with a distinguished element $a\in\A$ so that $\p = \p_a$, such that $\p_{a_N}(f) \to \p_a(f)$ for all $f\in\C\langle A,A^\ast\rangle$.
\end{definition}
\noindent Thus, Theorem \ref{thm 2} asserts that, in the case $\rho^N = \mu^N_{s,t}$, the mean empirical noncommutative distribution $\p^N_{s,t} = \E(\widetilde{\p}^N_{s,t})$ converges weakly, and moreover the empirical distribution converges weakly almost surely to the limit.  As these distributions are supported on invertible operators, the weak convergence statements hold on the larger class of ``test functions'' $f\in\C\langle A,A^{-1},A^\ast,A^{-\ast}\rangle$.

We now introduce extensions of $\C\langle A,A^\ast\rangle$ and $\C\langle A,A^{-1},A^\ast,A^{-\ast}\rangle$ that deserve to be called the {\bf universal enveloping algebras} of these spaces.  The reader is also directed to \cite[Section 3.4]{DHK2013}.

\begin{notation} \label{n.poly1} With $\EX$ and $\EX^+$ as in Notation \ref{n.EX1}, define
\begin{equation} \label{e.P1*} \PP = \C[\{v_\ex\}_{\ex\in\EX}] \qquad \text{and} \qquad \PP^+ = \C[\{v_\ex\}_{\ex\in\EX^+}] \subset \PP,  \end{equation}
the spaces of polynomials in the (commuting) indeterminates $v_\ex$.  Elements of these spaces are generally denoted $P,Q,R$; when emphasizing their variables, we write $P(\mx{v}) = P(\{v_\ex\})$.  For shorthand, we denote

\begin{equation} \label{e.vk} v_k = v_{\ex(k)}, \qquad k\in\Z\setminus\{0\}, \end{equation}
where $\ex(k) =( {\overset{k}{\overbrace{1,\dots,1}}})$ for $k>0$ and $\ex(k) =( {\overset{|k|}{\overbrace{-1,\dots,-1}}})$ for $k<0$.  Set $v_0\equiv 1$.  Define the subalgebra $\H\PP \subset\PP$ as follows:
\begin{equation} \label{e.P1old} \H\PP = \C[\{v_k\}_{k\in\Z\setminus\{0\}}]. \end{equation}
\end{notation}

\begin{remark}
In \cite{DHK2013}, $\PP$ was referred to as $\mathscr{W}$, while $\H\PP$ was simply denoted $\C[\mx{v}]$. \end{remark}

We may naturally identify $\C\langle A,A^\ast\rangle$ as a linear subspace of $\PP^+$, via the linear map
\begin{equation} \label{e.Upsilon} \Upsilon\colon\C\langle A,A^\ast\rangle\to \PP^+ \quad \text{ defined by } \quad \Upsilon(A^\ex) = v_\ex, \quad \ex\in\EX^+. \end{equation}
This is a complex vector space isomorphism from $\C\langle A,A^\ast\rangle$ onto $\mathrm{span}_\C\{v_\ex\colon\ex\in\EX^+\}$, the space of {\em linear} polynomials in $\PP^+$.  A similar identification could be made for $\C\langle A,A^{-1},A^\ast,A^{-\ast}\rangle$ in $\PP$, although for the inclusion to be well-defined and one-to-one we must restrict $\ex\in\EX$ to {\em reduced} words in the sense of $\F_2$; then $\Upsilon(\C\langle A,A^{-1},A^\ast,A^{-\ast}\rangle)$ is a strict subset of the linear polynomials in $\PP$.  Thus, if $\p$ is a linear functional on $\C\langle A,A^\ast\rangle$, it extends uniquely to a {\em homomorphism} $\PP^+\to\C$; in this sense, $\PP^+$ is the universal enveloping algebra of $\C\langle A,A^\ast\rangle$.  This will be useful in Section \ref{section trace polynomials}, and so we record this new role for $\p_a$ in the following notation.

\begin{notation} \label{n.poly2} Let $(\A,\t)$ be a noncommutative probability space. Let $\ex\in\EX^+$, and define $V_\ex\colon\A\to\C$ by $V_\ex = \p_{(\cdot)}(A^\ex)$:
\begin{equation} \label{e.Vex} V_\ex(a) = \t(a^\ex) = \t(a^{\ex_1}a^{\ex_2}\cdots a^{\ex_n}), \end{equation}
where $n=|\ex|$.   Let $\A^{\mathrm{inv}}$ denote the group of invertible elements in $\A$.  Then $V_\ex\colon\A^{\mathrm{inv}}\to\C$ is well-defined for any $\ex\in\EX$ by (\ref{e.Vex}), setting $a^{+\ast} \equiv a^\ast$ and $a^{-\ast}\equiv (a^\ast)^{-1} = (a^{-1})^\ast$. \end{notation}

\begin{remark} Strictly speaking, we should denote $V_\ex = V^{(\A,\t)}_\ex$ since this symbol represents different functions on different noncommutative probability spaces.  We will usually suppress this indexing, which will always be clear from context. \end{remark}

\subsection{Free Probability\label{section free prob}}

\begin{definition} \label{d.free} Let $(\A,\t)$ be a noncommutative probability space.  Unital subalgebras $\A_1,\ldots,\A_m\subset\A$ are called {\bf free} with respect to $\t$ if, given any $n\in\N$ and $k_1,\ldots,k_n\in\{1,\ldots,m\}$ such that $k_{i-1}\ne k_i$ for $1<i\le n$,  and any elements $a_i\in \A_{k_i}$ with $\t(a_i)=0$ for $1\le k\le n$, it follows that $\t(a_1\cdots a_n)=0$.  Random variables $a_1,\ldots,a_m$ are said to be {\bf freely independent} of the unital $\ast$-algebras $\A_i = \langle a_i,a_i^\ast\rangle\subset \A$ they generate are free.\end{definition}
Free independence is a $\ast$-moment factorization property.  By centering $a_i-\t(a_i)1_\A\in\A_i$, the freeness rule allows (inductively) any moment $\t(a_{k_1}^{\ex_1} \cdots a_{k_n}^{\ex_n})$ to be decomposed as a polynomial in moments $\t(a_i^\ex)$ in the variables separately.  In terms of Definition \ref{d.ncdist0} (which can be extended naturally to the multivariate case, see \cite[Lecture 4]{NicaSpeicherBook}), if $a_1,\ldots,a_m$ are freely independent then their joint noncommutative distribution $\p_{a_1,\ldots,a_n}$ is determined (computationally effectively) by the individual noncommutative distributions $\p_{a_1},\ldots,\p_{a_m}$.  

If $\A$ is a $W^\ast$-algebra and $a\in\A$ is normal, then $\p_a$ is completely described by a compactly-supported measure $\mu_a$ on $\C$; cf.\ (\ref{e.ncdist1'}).  Thus, if $u,v\in\A$ are freely independent unitary operators, $uv$ is also unitary, and the distributions $\mu_u$, $\mu_v$, and $\mu_{uv}$ are supported on $\U$.  Since $\mu_{uv}$ is determined by $\p_{u,v}$ which, by freeness, is determined by $\mu_u$ and $\mu_v$, there is a well-defined operation, {\bf free multiplicative convolution} $\boxtimes$, on probability measures on $\U$ such that $\mu_{uv} = \mu_u\boxtimes\mu_v$.  Similarly, if $x,y\in\A$ are positive definite, the distribution $\p_{xy}$ of their product is determined by the measures $\mu_x$ and $\mu_y$ supported in $\R_+$.  Although $xy$ is not necessarily normal, it is easy to check that it has the same noncommutative distribution as the positive definite operators $\sqrt{x}y\sqrt{x}$ and $\sqrt{y}x\sqrt{y}$.  So if we define $x\odot y = \sqrt{x}y\sqrt{x}$, then there is a well-defined operation $\boxtimes$ on probability measures on $\R_+$ such that $\mu_{x\odot y} = \mu_x\boxtimes\mu_y$; this is also called free multiplicative convolution.  In both frameworks, it can be described succinctly in terms of the {\bf $\Sigma$-transform}.

\begin{definition} \label{d.S-trans} Let $\mu$ be a probability measure on $\C$.  Define the function 
\[ \psi_\mu(z) = \int_{\C} \frac{\zeta z}{1-\zeta z}\,\mu(d\zeta), \quad z\notin\supp\mu, \]
which is analytic on its domain. If $\mu$ is supported in $\U$, it is customary to restrict $\psi_\mu$ to the unit disk $\mathbb{D}$; if $\mu$ is supported in $\R$, it is customary to restrict $\psi_\mu$ to the upper half-plane $\C_+$.  Define $\eta_\mu(z) = \psi_\mu(z)/(1+\psi_\mu(z))$.  This function is one-to-one on a neighborhood of $0$ if $\supp\mu\subset\U$ (and the first moment of $\mu$ is non-zero); it is one-to-one on the right-half plane $i\C_+$ if $\supp\mu\subset\R_+$; cf.\ \cite{Bercovici1993}.  The {\bf $\Sigma$-transform} $\Sigma_\mu$ is the analytic function
\begin{equation} \label{e.Sigma} \Sigma_\mu(z) = \frac{\eta_\mu^{-1}(z)}{z}, \end{equation}
for $z$ in a neighborhood of $0$ in the $\U$-case and for $z\in\eta_\mu(i\C_+)$ in the $\R_+$-case.
\end{definition}

The $\Sigma$-transform is a $\boxtimes$-homomorphism: as shown in \cite{Bercovici1992,Voiculescu1987},
\begin{equation} \label{e.Sigma1} \Sigma_{\mu\boxtimes\nu}(z) = \Sigma_{\mu}(z)\Sigma_{\nu}(z) \end{equation}
for any probability measures $\mu,\nu$ both supported in $\U$ (resp.\ $\R_+$), and any $z$ in a sufficiently small neighborhood of $0$ (resp.\ open set in $i\C_+$).

\begin{theorem}[Biane, 1997] \cite{Biane1997c} The measures $\{\nu_t\}_{t\in\R}$ of Definition \ref{d.nuNew} have $\Sigma$-transforms
\begin{equation} \label{e.Sigmanu} \Sigma_{\nu_t}(z) = e^{\frac{t}{2}\frac{1+z}{1-z}} \end{equation}
defined and analytic everywhere on $\C\setminus\{1\}$.  Hence, from (\ref{e.Sigma1}), they form a $\boxtimes$-group: for $s,t\in\R$, $\nu_{s+t} = \nu_s\boxtimes\nu_t$.
\end{theorem}

\begin{remark} In terms of the above discussion of free multiplicative convolution, $\nu_s\boxtimes\nu_t$ only makes sense if $st\ge 0$.  If, instead, we take (\ref{e.Sigma1}) as the definition of $\boxtimes$, then (\ref{e.Sigmanu}) shows the $\boxtimes$-group property holds for all $s,t\in\R$.
\end{remark}

Equation \ref{e.Sigmanu} was the starting point for investigation of the measures $\nu_t$ (with $t>0$).  In \cite[Lemmas 6.3 and 7.1]{Bercovici1992}, the authors showed that (\ref{e.Sigmanu}) defines a measure $\nu_t$ that is an analogue of the Gaussian on $\R$: it is the free multiplicative convolution power limit of a(n appropriately scaled) two-point measure.  Later, in \cite[Lemma 1]{Biane1997c}, Biane showed that these measures have the moments given in (\ref{e.nuNu0}). Using complex analytic techniques, a great deal of information can be gleaned about these measures.  The state of the art is summarized in the following proposition, where the $t>0$ statements were proved in \cite{Biane1997c}, while the $t<0$ case follows from results in \cite{Belinschi2004,Belinschi2005,Bercovici1992} and the recent preprint \cite{Zhong2013}.

\begin{proposition} For $t>0$, $\nu_t$ has a continuous density $\varrho_t$ with respect to the normalized Lebesgue measure on $\U$.  For $0<t<4$, its support is the connected arc
\[ \supp\nu_t = \left\{e^{i\theta}\colon -\frac12\sqrt{t(4-t)}-\arccos\left(1-\frac{t}{2}\right) \le \theta\le \frac12\sqrt{t(4-t)}+\arccos\left(1-\frac{t}{2}\right)\right\}, \]
while $\supp\nu_t=\U$ for $t\ge 4$.  The density $\varrho_t$ is real analytic on the interior of the arc.  It is symmetric about $1$, and is determined by $\varrho_t(e^{i\theta}) = \Re \kappa_t(e^{i\theta})$ where $z=\kappa_t(e^{i\theta})$ is the unique solution (with positive real part) to
\[ \frac{z-1}{z+1}e^{\frac{t}{2}z} = e^{i\theta}. \]

For $t<0$, $\nu_t$ has a continuous density $\varrho_t$ with respect to Lebesgue measure on $\R_+$.  The support is the connected interval $\supp\nu_t = [r_-(t),r_+(t)]$ where
\[ r_{\pm}(t) = \frac{2-t\pm\sqrt{t(t-4)}}{2}e^{-\frac12\sqrt{t(t-4)}}. \]
The density $\varrho_t$ is real analytic on the interval $(r_-(t),r_+(t))$, unimodal with peak at its mean $1$; it is determined by $\varrho_t(x) = \frac{1}{\pi x}\Im \zeta_t(x)$ where $z=\zeta_t(x)$ is the unique solution to
\[ \frac{z}{z-1}e^{-t\left(z-\frac12\right)} = x. \]
\end{proposition}

When $t>0$, the measure $\nu_t$ is the distribution of the {\bf free unitary Brownian motion} introduced in \cite{Biane1997c}.  The free unitary Brownian motion is a stationary, unitary-valued stochastic process $(u_t)_{t\ge0}$ such that the multiplicative increments $u_{t_1}, u_{t_2}u_{t_1}^\ast, \ldots,u_{t_n}u_{t_{n-1}}^\ast$ are freely independent for $0<t_1<t_2<\cdots<t_n<\infty$; up to a time-scaling factor, this implies that $\p_{u_t} = \nu_t$.  The process $u_t$ is constructed as the solution of a free stochastic differential equation.  Let $(\mathscr{A},\t)$ be a noncommutative probability space that contains a free semicircular Brownian motion $s_t$.  Then $u_t$ is defined to be the unique solution to the free SDE
\begin{equation} \label{e.fubmsde} du_t = iu_t\,ds_t - \frac12u_t\,dt \end{equation}
with $u_0=1$.  This precisely mirrors the matrix SDE satisfied by the Brownian motion on $\U_N$ (although the proof that $u_t$ is the noncommutative limit of this process does not follow easily from this observation).

For Section \ref{section proof of thm 2}, it will also be useful to consider the {\bf free multiplicative Brownian motion}, which is nominally the large-$N$ limit of the Brownian motion on $\GL_N$.  Let $(\mathscr{A},\tau)$ be a noncommutative probability space that contains two freely independent semicircular Brownian motions $s_t,s_t'$.  Then $c_t = \frac{1}{\sqrt{2}}(s_t+is_t')$ is called a {\bf circular Brownian motion}.  The free multiplicative Brownian motion $z_t$ is defined to be the unique solution to the free SDE
\begin{equation} \label{e.fmbmsde} dz_t = z_t\,dc_t \end{equation}
with $z_0=1$.  Again, this precisely mirrors the matrix SDE satisfied by the Brownian motion on $\GL_N$.  It was left as an open problem in \cite{Biane1997c} whether $z_t$ is the limit in noncommutative distribution of the $\GL_N$ Brownian motion.  The special case $s=t$ of Theorem \ref{thm 2} is a partial answer to this question.  In fact, using techniques similar to ours, the concurrent paper \cite{Guillaume2013} proves this full claim.  The reader is also directed to the author's papers \cite{Kemp2012,Kemp2012a} for detailed discussions of free stochastic calculus.

\section{Intertwining Operators and Concentration of Measure \label{section intertwine}}

In this section, we summarize the relevant results from the author's recent joint paper \cite{DHK2013}, in addition to giving some estimates of the involved constants.

\subsection{The Action of $\Delta_{\U_N}$ and $A^N_{s,t}$ on Trace Polynomials \label{section trace polynomials}}

If $Z\in\GL_N$, the noncommutative distribution $\p_Z$ (viewed as a homomorphism on $\PP$, as in Section \ref{section NC distributions}) induces a family of functions of $Z$: linear combinations of products of traces $\tr(Z^{\ex^{(1)}})\cdots \tr(Z^{\ex^{(m)}})$.  We call such functions {\bf trace polynomials}; cf.\ Notation \ref{n.poly3} below.  In this section, we will describe the action of the generalized Laplacian $A^N_{s,t}$ (and its special case $\Delta_{\U_N} = \left.A^N_{1,0}\right|_{\U_N}$) on trace polynomials.  We will rely heavily upon Notation \ref{n.poly1}, as well as the following.

\begin{notation} \label{n.poly2.5} Given $\ex^{(1)},\ldots,\ex^{(m)}\in\EX$, we say that the monomial $v_{\ex^{(1)}}\cdots v_{\ex^{(m)}}$ has {\bf trace degree} equal to  $|\ex^{(1)}|+\cdots+|\ex^{(m)}|$.  More generally, given any polynomial $P\in\PP$, the trace degree of $P$, denoted $\deg(P)$, is the highest trace degree among its monomial terms; if all terms have trace degree $n$, we say the polynomial has {\bf homogeneous} trace degree $n$.

For $n\in\N$, let $\PP_n = \{P\in\PP\colon \deg(P)\le n\}$.  Note that $\PP_n$ is finite-dimensional, $\PP_n\subset\C[\{v_\ex\}_{|\ex|\le n}]$, and $\PP = \bigcup_{n\ge 1} \PP_n$.  The sets $\H\PP_n$ are defined similarly.  In particular, $\H\PP_n \subset \C[v_{\pm 1},\ldots,v_{\pm n}]$, and, in terms of (\ref{e.vk}), this means
\[ \deg(v_1^{k_1}v_{-1}^{k_{-1}}\cdots v_n^{k_n}v_{-n}^{k_{-n}}) = \sum_{1\le |j|\le n} |j|k_j. \]
\end{notation}

\begin{notation} \label{n.poly3} Let $(\M_N)^{\EX}$ denote the set of functions $\EX\to\M_N$.  Denote by $\mx{V}_N$ the map $\GL_N\to (\M_N)^{\EX}$ given by
\[ [\mx{V}_N(Z)](\ex) = V_\ex(Z) = \tr(Z^\ex), \qquad Z\in\GL_N, \; \ex\in\EX. \]
For $P\in\PP$, we write $P\circ\mx{V}_N$ for the evaluation of $P$ as a function on $\GL_N$.  That is: if $\ex^{(1)},\ldots,\ex^{(n)}\in\EX$ are such that $P=P(v_{\ex^{(1)}},\ldots,v_{\ex^{(n)}})$ is in $\C[v_{\ex^{(1)}},\ldots,v_{\ex^{(n)}}]$, then
\[ (P\circ\mx{V}_N)(Z) = P(V_{\ex^{(1)}}(Z),\ldots,V_{\ex^{(n)}}(Z)). \]
We refer to any such function as a {\bf trace polynomial}.
\end{notation}
\noindent Note: in \cite{DHK2013}, the trace polynomial $P\circ\mx{V}_N$ was often denoted simply as $P_N$.

\begin{example} \label{ex.VN} If $P(\mx{v}) = v_{(1,\ast)}v_{(\ast)} + 2v_{(\ast,-1,1)}$ then $\deg(P) = 3$, and
\[ (P\circ\mx{V}_N)(Z) = \tr(ZZ^\ast)\tr(Z^\ast) + 2\tr(Z^\ast Z^{-1}Z) = \tr(ZZ^\ast)\tr(Z^\ast) + 2\tr(Z^\ast). \]
Thus, if we set $Q(\mx{v}) = v_{(1,\ast)}v_{(\ast)} + 2v_{(\ast)}$, then $P\circ\mx{V}_N = Q\circ\mx{V}_N$ for all $N$.  That is, the map $P\mapsto P\circ\mx{V}_N$ from $\PP$ to the space of trace polynomials is not one-to-one for any $N$.  If we restrict this map to $\H\PP$, cancellations like this do not occur; nevertheless, the map is still not one-to-one, due to the Cayley-Hamilton theorem, as explained in \cite[Section 2.4]{DHK2013}.  Nevertheless, restricted to $\H\PP_n$ for some $n\in\N$, the map {\em is} one-to-one for all sufficiently large $N$ (depending on $n$).
\end{example}

\begin{remark} Note that, if $P\in\H\PP$, then the function $P\circ\mx{V}_N$ is holomorphic on $\GL_N$.  This is the reason we use the notation $\H\PP$. \end{remark}

We now introduce two families of polynomials $\{Q^\pm_\ex\colon\ex\in\EX\}$ and $\{R^\pm_{\ex,\d}\colon \ex,\d\in\EX\}$ in $\PP$ that were introduced in \cite[Theorem 3.12]{DHK2013}.  Since we do not need to know all the details about these polynomials, the following is only as precise as will be needed below (in particular in Proposition \ref{prop Cnst bound}).

\begin{definition} \label{d.qr} Let $\ex\in\EX$, and let $1\le j<k\le |\ex|$. Define $n_{\pm}(\ex)$ be the integer from \cite[Eq.\ (3.36)]{DHK2013}; in particular, $|n_\pm(\ex)|\le |\ex|$, and let $\{\ex_{j,k}^\ell\colon \ell=0,1,2\}$ be the substrings of $\ex$ given in \cite[Eq.\ (3.37)]{DHK2013}; in particular, $\ex=\ex^0_{j,k}\ex^1_{j,k}\ex^2_{j,k}$ and so $|\ex^0_{j,k}|+|\ex^1_{j,k}|+|\ex^2_{j,k}|=|\ex|$.  Define
\begin{equation} \label{e.d.q} Q_\ex^\pm(\mx{v}) = n_\pm(\ex)v_\ex + 2\sum_{1\le j<k\le n} \pm v_{\ex^0_{j,k}\ex^2_{j,k}}v_{\ex^1_{j,k}}, \end{equation}
where the $\pm$ signs inside the sum depend on $\ex,j,k$.  For $s,t\in\R$, define
\begin{equation} \label{e.d.qst} Q^{s,t}_\ex = \left(s-\frac{t}{2}\right)Q^+_\ex + \frac{t}{2}Q^-_{\ex}. \end{equation}
Thus, except when $(s,t)=(0,0)$, $Q^{s,t}_{\ex}$ is a homogeneous trace degree $|\ex|$ polynomial.

Additionally, let $\d\in\EX$.  For $1\le j\le |\ex|$ and $1\le k\le |\d|$, let $\ex^{(j)}$ and $\d^{(k)}$ be the cyclic permutations of $\ex$ and $\d$ in \cite[Eq.\ (3.40)]{DHK2013}.  Define
\begin{equation} \label{e.d.r} R^\pm_{\ex,\d}(\mx{v}) = \sum_{j=1}^{|\ex|}\sum_{k=1}^{|\d|} \pm v_{\ex^{(j)}\d^{(k)}}, \end{equation}
where the $\pm$ signs inside the sum depend on $\ex,\d,j,k$.  For $s,t\in\R$, define
\begin{equation} \label{e.d.rst} R^{s,t}_{\ex,\d} = \left(s-\frac{t}{2}\right)R^+_{\ex,\d} + \frac{t}{2}R^-_{\ex,\d}. \end{equation}
Thus, except when $(s,t)=(0,0)$, $R^{s,t}_{\ex,\d}$ is a homogeneous trace degree $|\ex|+|\d|$ polynomial.

\end{definition}

The following {\bf intertwining formulas} were the core computational tools in \cite{DHK2013}.

\begin{theorem}[Intertwining Formulas] {\em \cite[Theorems 1.20 \& 3.13]{DHK2013}} \label{thm intertwines} Let $s,t\in\R$. Define the following differential operators on $\PP$:
\begin{equation} \label{e.d.DstLst} \D_{s,t} = \frac12\sum_{\ex\in\EX} Q^{s,t}_{\ex}(\mx{v}) \frac{\del}{\del v_\ex} \qquad \text{and} \qquad
\L_{s,t} = \frac12\sum_{\ex,\d\in\EX} R^{s,t}_{\ex,\d}(\mx{v})\frac{\del^2}{\del v_\ex\del v_\d}, \end{equation}
where $Q^{s,t}_\ex$ and $R^{s,t}_{\ex,\d}$ are as in Definition \ref{d.qr}.  Then for any $P\in\PP$, we have
\begin{equation}\label{e.intertwine0} \frac12A^N_{s,t} (P\circ \mx{V}_N) = -\left[\D_{s,t}P + \frac{1}{N^2}\L_{s,t}P\right]\circ \mx{V}_N. \end{equation}
In the special case $(s,t)=(1,0)$,
\begin{align}
\label{e.intertwine1} \left.\D_{1,0}\right|_{\H\PP} &= \frac12\sum_{|k|\ge 1} |k|v_k\frac{\del}{\del v_k} + \frac12\sum_{k= 2}^\infty k \left[\left(\sum_{j=1}^{k-1} v_jv_{k-j}\right)\frac{\del}{\del v_k} + \left(\sum_{j=1}^{k-1} v_{-j}v_{-(k-j)}\right)\frac{\del}{\del v_{-k}}\right],  \\
\label{e.intertwine2} \left.\L_{1,0}\right|_{\H\PP} &= \frac12\sum_{|j|,|k|\ge 1} jkv_{j+k}\frac{\del^2}{\del v_j \del v_k}.
\end{align}
\end{theorem}

\begin{notation} For $N\ge 1$, we set
\begin{equation} \label{e.DN} \D^N_{s,t} = \D_{s,t}+ \frac{1}{N^2}\L_{s,t}. \end{equation}
Thus (\ref{e.intertwine0}) asserts that $\frac12A^N_{s,t}(P\circ\mx{V}_N) = -[\D^N_{s,t}P]\circ\mx{V}_N$.
\end{notation}

\begin{remark}\begin{itemize}
\item[(1)] In the notation of \cite[Definition 1.16]{DHK2013}, $\left.\D_{1,0}\right|_{\H\PP} = \mathcal{N}_0+2\mathcal{Z}$ (rewritten here using the trick of Remark 5.13 in that paper).  Note, also, that the terms with $j=-k$ in (\ref{e.intertwine2}) involve $v_0$, which we interpret as $1$.
\item[(2)] Since $\Delta_{\U_N} = \left.A_{1,0}^N\right|_{\U_N}$, (\ref{e.intertwine0}) shows that
\begin{equation} \label{e.intertwine0'} \frac12\Delta_{\U_N}(P\circ\mx{V}_N) = -\left[\D_{1,0}P + \frac{1}{N^2}\L_{1,0}P\right]\circ\mx{V}_N = -[\D_{1,0}^N P]\circ \mx{V}_N. \end{equation}
This is the formal sense in which (\ref{e.int0}) is true.  For a trace polynomial $\left.(P\circ \mx{V}_N)\right|_{\U_N}$ with $P\in\H\PP$, the Laplacian can be calculated explicitly using (\ref{e.intertwine1}) and (\ref{e.intertwine2}).
\end{itemize}
\end{remark}

\begin{example} \label{ex.quadtrace} Consider the trace polynomials $f(U) = \tr(U^n)\tr(U^m)$ for $U\in\U_N$; for convenience we assume $n,m\ge 2$.  Then $f = P\circ\mx{V}_N$ where $f(\mx{v}) = v_nv_m \in \H\PP_+$.  Then (\ref{e.intertwine1}) and (\ref{e.intertwine2}) give
\begin{align} 2\D_{1,0}(v_nv_m) &= (n+m)v_nv_m + n\sum_{j=1}^{n-1}v_jv_{n-j}v_m + m\sum_{j=1}^{m-1} v_jv_{m-j}v_n, \\
2\L_{1,0}(v_nv_m) &= 2nmv_{n+m}.
\end{align}
Note that all terms have homogeneous trace degree $n+m$, the same as $v_nv_m$; this follows from Theorem \ref{thm intertwines}.  Thus, (\ref{e.intertwine0'}) yields
\begin{align*} \Delta_{\U_N}\left(\tr(U^n)\tr(U^m)\right) &= -(n+m)\tr(U^n)\tr(U^m) -\frac{2nm}{N^2}\tr(U^{n+m}) \\
&\qquad - n\sum_{j=1}^{n-1}\tr(U^j)\tr(U^{n-j})\tr(U^m) - m\sum_{j=1}^{m-1}\tr(U^j)\tr(U^{m-j})\tr(U^n). \end{align*}
In the special case $N=1$, $\tr(U^j) = U^j$, and so the calculation shows that
\[ \Delta_{\U_1} (U^{n+m}) = -(n+m)U^{n+m}-2nmU^{n+m}-n\sum_{j=1}^{n-1} U^{n+m} - m\sum_{j=1}^m U^{n+m} = -(n+m)^2U^{n+m}, \]
which is consistent with (\ref{e.Delta1}).
\end{example}

We record here another intertwining formula (that did not appear in \cite{DHK2013}) regarding the complex conjugation map.

\begin{definition} \label{d.conj} Given $\ex\in\EX$, define $\ex^\ast\in\EX$ by $(\ex_1,\ldots,\ex_n)^\ast = (\ex_n^\ast,\ldots,\ex_1^\ast)$, where $(\pm 1)^\ast = \pm \ast$ and $(\pm \ast)^\ast = \pm 1$.  Let $\CC\colon\PP\to\PP$ be the conjugate linear homomorphism defined by $\CC(v_\ex) = v_{\ex^\ast}$ for all $\ex\in\EX$.  Note that, for any $P\in\PP$ and $Z\in\Z_N$,
\begin{equation} \label{e.intertwineC} \overline{P\circ\mx{V}_N(Z)} = (\CC P\circ\mx{V}_N)(Z). \end{equation}
That is: $\CC$ intertwines complex conjugation.  This follows from the fact that $\overline{\tr(Z^\ex)} = \tr(Z^{\ex^\ast})$.  We will often write $\CC(P) = P^\ast$.
\end{definition}

\begin{lemma} \label{l.[C,A]} The complex conjugation intertwiner $\CC$ commutes with the operators $\D_{s,t}$, $\L_{s,t}$, and hence $\D^N_{s,t}$. \end{lemma}

\begin{proof} Fix $N\in\N$ and let $P\in\PP$ and $Z\in \M_N$.  From Remark \ref{remark must}(3) and (\ref{e.intertwineC}), together with (\ref{e.intertwine0}) and (\ref{e.DN}), we have
\[ \big(e^{-\D_{s,t}^N} \CC P\big)\circ\mx{V}_N\,(Z) = \big(e^{\frac12A^N_{s,t}}(\CC P\circ\mx{V}_N)\big)(Z) = \overline{\big(e^{\frac12A^N_{s,t}}(P\circ\mx{V}_N)\big)(Z)} = \CC\big(e^{-\D^N_{s,t}}P)\circ\mx{V}_N(Z). \]
It follows that
\[ \big([\CC,e^{-\D^N_{s,t}}]P\big)(Z) = 0, \qquad N\in\N,\; Z\in\GL_N. \]
It follows from \cite[Theorem 2.10]{DHK2013} (asymptotic uniqueness of trace polynomial representations) that the polynomial $[\CC,e^{-\D^N_{s,t}}]P=0$.  Scaling $(s,t)\mapsto (\alpha s,\alpha t)$ and differentiating with respect to $\alpha$ at $\alpha=0$ shows that $[\CC,\D^N_{s,t}]P=0$.  As this holds for each $N$, sending $N\to\infty$ (using continuity of all involved maps on the finite-dimensional $\D^N_{s,t}$-invariant subspace of polynomials with trace degree $\le \deg(P)$) shows that $[\CC,\D_{s,t}]P=0$, and it then follows that $[\CC,\L_{s,t}]P=0$.  Since these hold for all $P\in\PP$, the lemma is proved.  \end{proof}

\begin{remark} It is possible to prove Lemma \ref{l.[C,A]} with direct computation from the definitions (\ref{e.d.DstLst}) of the intertwining operators $\D_{s,t}$ and $\L_{s,t}$; the proof we've given is much shorter. 
\end{remark}

As noted in Example \ref{ex.quadtrace}, the operators $\D_{s,t}$ and $\L_{s,t}$ in Theorem \ref{thm intertwines} preserve trace degree (so long as $(s,t)\ne(0,0)$).  Hence, so do the operators $\D^N_{s,t}$ which intertwine $-\frac12A^N_{s,t}$.  In particular, this means that, for each $n\in\N$, $\PP_n$ is an invariant subspace for $\D^N_{s,t}$; equivalently, by (\ref{e.intertwine0}), the finite-dimensional subspace $\PP_n\circ\mx{V}^N$ of trace polynomials ``of trace degree $\le n$'' is an invariant subspace for $A^N_{s,t}$.  (Note: from the second term in $P$ in Example \ref{ex.VN}, we see that trace degree is not well-defined for trace polynomial functions, only for their intertwining polynomials.  However, the subspace $\PP_n\circ\mx{V}_N$ is a well-defined, finite-dimensional invariant subspace for $A^N_{s,t}$.)

Let $n\in\N$.  The restriction $\left.\D^N_{s,t}\right|_{\PP_n}$ is a finite-dimensional operator, and so can be exponentiated in the usual manner.  Similar considerations applied to $\left.A^N_{s,t}\right|_{\PP_n\circ\mx{V}_N}$, together with (\ref{e.intertwine0}), show that
\begin{equation} \label{e.int3} e^{\frac12A^N_{s,t}}(P\circ\mx{V}_N) = \big(e^{-\D^N_{s,t}}P\big)\circ\mx{V}_N, \qquad P\in \PP, \end{equation}
where the restrictions are done with $n=\deg(P)$.  Combining this with (\ref{e.hkm2}) shows that, for $s,t>0$ with $s>t/2$,
\begin{equation} \label{e.hkm2'} \E_{\mu_{s,t}^N}(P\circ\mx{V}_N) = \big(e^{-\D^N_{s,t}}P\big)(\mx{1}), \end{equation}
where by $P(\mx{1})$ we mean the complex number given by setting all $v_\ex=1$ in $P(\mx{v})$.  Analogous considerations from (\ref{e.hkm1}) and (\ref{e.intertwine0'}) show that, for $t>0$,
\begin{equation} \label{e.hkm1'} \E_{\rho_t^N}(P\circ\mx{V}_N) = \big(e^{-\D^N_{t,0}}P\big)(\mx{1}). \end{equation}

\subsection{Concentration of Heat Kernel Measure\label{section concentration hkm}}

The expectation-intertwining formulas (\ref{e.hkm1'}) and (\ref{e.hkm2'}) show there is $O(1/N^2)$-concentration of the $\U_N$ or $\GL_N$ heat kernel measure's mass.  The following lemma makes this precise. It is a version of \cite[Lemma 4.1]{DHK2013}; we expand on the statement and proof here to give some quantitative estimates (cf.\ Proposition \ref{prop Cnst bound}).

\begin{lemma} \label{l.findim} Let $V$ be a finite dimensional normed $\C$-space.  For parameters $s,t\in\R$, let $D_{s,t}$ and $L_{s,t}$ be two operators on $V$ that depend continuously on $s$ and $t$.  Then there exists a constant $C(s,t)<\infty$, depending continuously on $(s,t)\in \R^2$, such that, for any linear functional $\psi\colon V\to\C$,
\begin{equation} \label{e.findim} \left|\psi(e^{D_{s,t}+\e L_{s,t}}x)-\psi(e^{D_{s,t}}x)\right| \le C(s,t)\|x\|_V\|\psi\|_{V^\ast}|\e|, \qquad x\in V, \; |\e|\le 1. \end{equation}
\end{lemma}
\noindent Note that the constant $C(s,t)$ also depends on the norm $\|\cdot\|_V$.
\begin{proof} We follow our proof in \cite[Lemma 4.1]{DHK2013}.  For the moment, write $D=D_{s,t}$ and $L=L_{s,t}$.  Using the well known differential of the exponential map (see for 
example \cite[Theorem 1.5.3, p.\ 23]{Duistermaat2000} or \cite[Theorem 3.5, p.\ 70]{HallLieBook}),
\begin{align*}
\frac{d}{du}e^{D+uL}  &  =e^{D+uL}\int_{0}^{1}e^{-v\left(  D+uL\right)
}Le^{v\left(  D+uL\right)  }dv \\
&  =\int_{0}^{1}e^{\left(  1-v\right)  \left(  D+uL\right)  }Le^{v\left(
D+uL\right)  }dv,
\end{align*}
we may write
\[
e^{D+\epsilon L}-e^{D}=\int_{0}^{\epsilon}\frac{d}{du}e^{D+uL}
du=\int_{0}^{\varepsilon}\left[  \int_{0}^{1}e^{\left(  1-v\right)  \left(
D+uL\right)  }Le^{v\left(  D+uL\right)  }dv\right]  du.
\]
Crude bounds now show
\begin{equation} \label{e.crude}
\left\Vert e^{D+\epsilon L}-e^{D}\right\Vert _{\mathrm{End}(V)}\leq\int_{0}^{|\e|}\left[  \int_{0}^{1}\left\Vert e^{(1-v)(D+uL)}Le^{v(D+uL)}\right\Vert_{\mathrm{End}(V)}dv\right]  du, \end{equation}
where $\|\cdot\|_{\mathrm{End}(V)}$ is the operator norm induced by $\|\cdot\|_V$.  Now, set
\begin{equation} \label{e.Cst1} C(s,t) = \sup_{0\le u\le|\e|\atop 0\le v\le 1}\left\| e^{(1-v)(D_{s,t}+uL_{s,t})}L_{s,t}e^{v(D_{s,t}+uL_{s,t})}\right\|_{\mathrm{End}(V)}. \end{equation}
(This constant nominally depends on $\e$, but we can take $\e=1$ here to provide a uniform bound.)  The function $(u,v,s,t)\mapsto e^{(1-v)(D_{s,t}+uL_{s,t})}L_{s,t}e^{v(D_{s,t}+uL_{s,t})}$ is continuous, and hence $C(s,t)$ is a continuous in $(s,t)$.  
Equations (\ref{e.crude}) and (\ref{e.Cst1}) show that
\begin{equation} \label{e.Cst4} \left\Vert e^{D+\epsilon L}-e^{D}\right\Vert _{\mathrm{End}(V)} \le C(s,t)|\e|; \end{equation}
and (\ref{e.findim}) follows immediately from (\ref{e.Cst4}). \end{proof}

Since $\psi(P) = P(\mx{1})$ defines a linear functional on $\PP_n$ for each $n$, (\ref{e.hkm1'}), (\ref{e.hkm2'}), and Lemma \ref{l.findim} immediately yield the following.

\begin{corollary} \label{cor.conc} For $s,t\in\R$ and $P\in\PP$, there is a constant $C(s,t,P)<\infty$, continuous in $(s,t)\in\R^2$, so that
\begin{equation} \label{e.conc1} \left|\big(e^{-\D^N_{s,t}}P\big)(\mx{1}) - \big(e^{-\D_{s,t}}P\big)(\mx{1})\right| \le \frac{1}{N^2}\cdot C(s,t,P). \end{equation}
\end{corollary}

\begin{proof} Let $n=\deg P$, and choose any norm $\|\cdot\|_{\PP_n}$ on the finite-dimensional space $\PP_n$; then $C(s,t,P)$ can be taken to equal $C(s,t)\|\psi\|_{\PP_n}^\ast\|P\|_{\PP_n}$ where $\psi(P)=P(\mx{1})$ and the constant $C(s,t)$ is from (\ref{e.Cst1}) with the operators $D_{s,t} = -\left.\D_{s,t}\right|_{\PP_n}$ and $L_{s,t} = -\left.\L_{s,t}\right|_{\PP_n}$.
\end{proof}

Corollary \ref{cor.conc} (in the special case $(s,t)\mapsto (t,0)$) shows that the large-$N$ limit of the heat kernel expectation $\E_{\rho^N_t}$ of any trace polynomial is given by the flow operator $e^{-\mathcal{D}_{t,0}}$; in this sense, $\D_{1,0}$ is the generator of the limit heat kernel (and hence of the free unitary Brownian motion).  In particular, taking $P = v_n$ so that $(P\circ\mx{V}_N)(U) = \tr(U^n)$, (\ref{e.hkm1'}) and (\ref{e.conc1}) show that
\begin{equation} \label{e.D-nu} \big(e^{-\mathcal{D}_{t,0}}v_k\big)(\mx{1}) = \lim_{N\to\infty}\int_{\U_N} \tr(U^n)\,\rho_t^N(dU) = \nu_k(t) \end{equation}
are the moments of $\nu_t$; cf.\ Definition \ref{d.nuNew}.  Since $\mathcal{D}_{t,0}$ is a first-order differential operator, the semigroup $e^{-\mathcal{D}_{t,0}}$ is an algebra homomorphism, and since the evaluation-at-$\mx{1}$-map is also a homomorphism, the complete description of the semigroup acting on $\H\PP$ is given by
\begin{equation} \label{e.D-nu1} \left(e^{-\mathcal{D}_{t,0}}(v_1^{k_1}v_{-1}^{k_{-1}}\cdots v_n^{k_n}v_{-n}^{k_{-n}})\right)(\mx{1}) = \nu_1(t)^{k_1}\nu_{-1}(t)^{k_{-1}}\cdots \nu_n(t)^{k_n}\nu_{-n}(t)^{k_{-n}}. \end{equation}
This simplifies further, since $\nu_{-m}(t) = \nu_m(t)$ for all $m$.

\subsection{Estimates on the Constants $C(s,t,P)$ \label{section C(s,t,P)}}

Corollary \ref{cor.conc} suffices to prove weak a.s.\ convergence of distributions when using (Laurent) polynomial test functions; in particular, this will suffice to prove Theorem \ref{thm 2}.  To extend this convergence to a larger class of test functions, as in Theorems \ref{thm 1}--\ref{thm 1.2}, we will need some quantitative information about the constants $C(s,t,P)$ in (\ref{e.conc1}).  To prove such estimates, we begin by introducing a norm on $\PP$ that will be used throughout the remainder of this section.

\begin{definition} \label{d.ell1norm} Let $\|\cdot\|_1$ denote the $\ell^1$-norm on $\PP$.  Precisely: let $\N^{\EX}_0$ denote the set of functions $\mx{k}\colon\EX\to\N$ that are finitely-supported.  For $\mx{k}\in\N^{\EX}_0$, define $\mx{v}^{\mx{k}}$ to be the monomial 
\[ \mx{v}^{\mx{k}} = \prod_{\ex\in\supp\mx{k}} v_\ex^{\mx{k}(\ex)}. \]
Any $P\in\PP$ has a unique representation of the form
\begin{equation} \label{e.Prep1} P(\mx{v}) = \sum_{\mx{k}\in\N_0^{\EX}} a_\mx{k} \mx{v}^{\mx{k}} \end{equation}
for some coefficients $a_{\mx{k}}\in\C$ that are $0$ for all but finitely-many $\mx{k}$.  Then we define
\begin{equation} \label{e.Pnorm} \|P\|_1 = \sum_{\mx{k}\in\N_0^{\mx{k}}} |a_\mx{k}|. \end{equation}
The uniqueness of the representation (\ref{e.Prep1}) of $P$ shows that $\|\cdot\|_1$ is well-defined on $\PP$, and it is easily verified to be a norm.
\end{definition}

We will use the norm $\|\cdot\|_1$ of (\ref{e.Pnorm}) to provide concrete bounds on $C(s,t,P)$ for $P\in \PP_n$; this will suffice to prove Theorems \ref{thm 1.1} and \ref{thm 1.2} (as well as a weaker version of Theorem \ref{thm 1}, with ultra-analytic test functions). We remind the reader of the following lemma: the operator norm on matrices induced by the $\ell^1$-norm is bounded by the maximal column sum of the matrix argument.

\begin{lemma} \label{lemma colsum} Let $V$ be a finite dimensional vector space, and let $e_1,\ldots,e_d\in V$ be a basis.  Let $\|\cdot\|_1$ denote the norm $\|a_1e_1+\cdots+a_de_d\|_1 = |a_1|+\cdots+|a_d|$ on $V$.  Then for $A\in\mathrm{End}(V)$, the operator norm $\displaystyle{\|A\|_{1\to 1} = \sup_{\|w\|_1=1} \|Aw\|_1}$ is bounded by
\begin{equation} \label{e.colsum0} \|A\|_{1\to 1} \le \max_{1\le j\le d} \|A(e_j)\|_1. \end{equation}
\end{lemma}

\begin{proof} Letting $w=a_1e_1+\cdots+a_de_d$, compute
\begin{align*} \|Aw\|_1 = \| a_1 A(e_1) + \cdots + a_d A(e_d)\|_1 \le \sum_{k=1}^d |a_k| \|A(e_k)\|_1 \le \max_{1\le j\le d} \|A(e_j)\|_1 \sum_{k=1}^d|a_k|, \end{align*}
and since $\sum_{k=1}^d|a_k| = \|w\|_1$, this proves the result. \end{proof}

\begin{remark} \label{r.basis} If we represent a vector in $V$ in a non-unique way, for example $v = a_1e_1 + a_2e_2 + b_1e_1 = (a_1+b_1)e_1+a_2e_2$, note that $\|v\|_1 = |a_1+b_1|+|a_2| \le |a_1|+|b_1|+|a_2|$; thus, if we use such a redundant representation for a vector when ``computing'' the $\|\cdot\|_1$-norm, we will always get an upper bound.  This will be relevant in the proof of Proposition \ref{prop Cnst bound} below, where it will be challenging to detect repeated occurrences of basis vectors. \end{remark}

We now prove a quantitative bound for the constants $C(s,t,P)$ for any $P\in\PP$.

\begin{proposition} \label{prop Cnst bound} Let $s,t\in\R$, let $n\in\N$, and let $P\in\PP_n$.  Define $r= |s-\frac{t}{2}|+\frac12|t|$.  Then for all $N\ge 1$,
\begin{equation} \label{e.Cnst bound}  \left|\big(e^{-\D^N_{s,t}}P\big)(\mx{1}) - \big(e^{-\D_{s,t}}P\big)(\mx{1})\right| \le \frac{1}{N^2}\cdot\frac{r}{2}n^2 e^{\frac{r}{2}n^2\left(1+\frac{1}{N^2}\right)}\|P\|_1. \end{equation}
\end{proposition}

\begin{proof} Let $V = \PP_n$ equipped with the norm $\|\cdot\|_1$ of (\ref{e.Pnorm}), let $\psi(P) = P(\mx{1})$, and set $D = -\mathcal{D}_{s,t}$ and $L = -\mathcal{L}_{s,t}$.  Then Lemma \ref{l.findim} shows that
\begin{equation} \label{e.borg0} \left|\big(e^{-\D^N_{s,t}}P\big)(\mx{1}) - \big(e^{-\D_{s,t}}P\big)(\mx{1})\right| = \left|\psi(e^{D+\frac{1}{N^2}L}P)-\psi(e^{D}P)\right| \le \frac{1}{N^2}C\|\psi\|_1^\ast \|P\|_1, \end{equation}
where
\begin{equation} \label{e.borg1} C = \sup_{0\le u\le 1/N^2\atop 0\le v\le 1}\left\| e^{(1-v)(D+uL)}Le^{v(D+uL)}\right\|_{1\to 1}. \end{equation}
Note that, for $P(\mx{v}) = \sum_{\mx{k}} a_{\mx{k}}\mx{v}^{\mx{k}}$ as in (\ref{e.Prep1}),
\begin{equation} \label{e.psi0} |\psi(P)| = |P(\mx{1})| = \Big|\sum_{\mx{k}} a_{\mx{k}}\Big| \le \|P\|_1, \quad \text{and therefore} \quad \|\psi\|_1^\ast \le 1. \end{equation}
Hence, to prove the proposition, it suffices to show that (\ref{e.borg1}) is bounded by $\frac{s}{2}n^2e^{\frac{s}{2}n^2(1+1/N^2)}$.

Since the operator norm $\|\cdot\|_{1\to1}$ is submultiplicative, for $0\le u,v\le 1$ we can estimate
\begin{align*} \left\| e^{(1-v)(D+uL)}L_{t,0}e^{v(D+uL)}\right\|_{1\to 1} &\le  \left\| e^{(1-v)(D+uL)}\right\|_{1\to 1}\cdot \left\| e^{v(D+uL)}\right\|_{1\to 1}\cdot \|L\|_{1\to1} \\
&\le  e^{(1-v)(\|D\|_{1\to1}+u\|L\|_{1\to1})}\cdot e^{v(\|D\|_{1\to1}+u\|L\|_{1\to1})}\cdot \|L\|_{1\to1} \\
& = e^{\|D\|_{1\to1}}\cdot e^{u\|L\|_{1\to 1}}\cdot \|L\|_{1\to1}
\end{align*}
where the second line follows from expanding the power series of the exponentials and repeatedly using the triangle inequality and submultiplicativity of the norm $\|\cdot\|_{1\to1}$.  Hence, taking the supremum over $0\le u\le 1/N^2$, we have
\begin{equation} \label{e.borg2} C \le  e^{\|D\|_{1\to1}}\cdot e^{\frac{1}{N^2}\|L\|_{1\to 1}}\cdot \|L\|_{1\to1}. \end{equation}
It behooves us to estimate $\|L\|_{1\to 1}$ and $\|D\|_{1\to 1}$; we do this using Lemma \ref{lemma colsum}.

The basis of $\PP_n$ defining the norm $\|\cdot\|_1$ is the set of monomials in $\PP_n$; that is, using the notation of Definition \ref{d.ell1norm}, the basis is
\[ \B_n = \{\mx{v}^{\mx{k}}\colon \deg(\mx{v}^\mx{k})\le n\} = \{\mx{v}^{\mx{k}}\colon \sum_{\ex\in\EX}|\mx{k}(\ex)||\ex|\le n\}. \]
We must therefore estimate the $\|\cdot\|_1$-norm of the images of $D = -\D_{s,t}$ and $L=-\L_{s,t}$ on these basis vectors.  So, fix a finitely-supported function $\mx{k}\colon\EX\to\N$.  Then for any $\ex\in\EX$, we have
\[ \frac{\del}{\del v_{\ex}}\mx{v}^{\mx{k}} = \mx{k}(\ex) \frac{\mx{v}^\mx{k}}{v_\ex}, \quad \text{where} \quad \frac{\mx{v}^\mx{k}}{v_\ex} \in\B_n. \]
(I.e.\ we write $\frac{\del}{\del v}v^k = kv^{k-1} = kv^k/v$ to simplify notation.)  Thus, from (\ref{e.d.q}), we have
\[ \sum_{\ex\in\EX} Q^{\pm}_\ex(\mx{v})\frac{\del}{\del v_\ex}\mx{v}^\mx{k} =  \sum_{\ex\in\EX}\mx{k}(\ex) \left[n_{\pm}(\ex)\mx{v}^\mx{k} + 2\sum_{1\le j<k\le |\ex|} \pm \frac{v_{\ex^0_{j,k}\ex^2_{j,k}}v_{\ex^1_{j,k}}}{v_\ex}\mx{v}^\mx{k}\right]. \]
(This is a finite sum: $\mx{k}(\ex)=0$ for all but finitely-many $\ex\in\EX$.)  Thus, from (\ref{e.d.qst}) and (\ref{e.d.DstLst}), we have
\begin{align*} \D_{s,t}(\mx{v}^\mx{k}) &= \frac12\sum_{\ex\in\EX}\mx{k}(\ex)\left[\left(s-\frac{t}{2}\right)n_+(\ex)+\frac{t}{2}n_-(\ex)\right] \cdot\mx{v}^\mx{k} \\
& \qquad + \sum_{\ex\in\EX}\mx{k}(\ex)\sum_{1\le j<k\le |\ex|}\left[ \left(s-\frac{t}{2}\right)(\pm 1) + \frac{t}{2}(\pm)\right]\frac{v_{\ex^0_{j,k}\ex^2_{j,k}}v_{\ex^1_{j,k}}}{v_\ex}\mx{v}^\mx{k}.
\end{align*}
All of the vectors $\mx{v}^{\mx{k}}$ and $v_{\ex^0_{j,k}\ex^2_{j,k}}v_{\ex^1_{j,k}}\mx{v}^\mx{k}/v_\ex$ in the above sum are basis vectors in $\B_n$.  They may not be distinct, but by Remark \ref{r.basis} we can compute an upper bound for the norm by simply summing the absolute values of the coefficients:
\begin{align*} \|\D_{s,t}(\mx{v}^\mx{k})\|_1 &\le  \frac12\sum_{\ex\in\EX}\mx{k}(\ex)\left|\left(s-\frac{t}{2}\right)n_+(\ex)+\frac{t}{2}n_-(\ex)\right| \\
& \qquad + \sum_{\ex\in\EX}\mx{k}(\ex)\sum_{1\le j<k\le |\ex|}\left| \left(s-\frac{t}{2}\right)(\pm 1) + \frac{t}{2}(\pm)\right|.
\end{align*}
We can estimate the internal terms as follows: since $|n_\pm(\ex)| \le |\ex|$ (cf.\ Definition \ref{d.qr}),
\[ \left|\left(s-\frac{t}{2}\right)n_+(\ex)+\frac{t}{2}n_-(\ex)\right| \le \left|s-\frac{t}{2}\right||n_+(\ex)| + \frac{1}{2}|t||n_-(\ex)| \le r|\ex| \]
and similarly the term inside the double sum is $\le r$.  Hence, we have
\begin{equation} \label{e.Dst.est1} \|\D_{s,t}(\mx{v}^\mx{k})\|_1 \le \frac{r}{2}\sum_{\ex\in\EX}|\ex|\mx{k}(\ex) + r\sum_{\ex\in\EX}\mx{k}(\ex)\frac{|\ex|(|\ex|-1)}{2}. \end{equation}
Since $\mx{v}^\mx{k}\in\B_n$, we have $\sum_{\ex\in\EX}|\ex|\mx{k}(\ex) \le n$, and so too $|\ex|\le n$ for any nonzero term in the sum.  Thus, (\ref{e.Dst.est1}) yields
\begin{equation} \label{e.D.est1} \|D\|_{1\to1} \le \max_{\mx{v}^\mx{k}\in\B_n} \|-\D_{s,t}(\mx{v}^\mx{k})\|_1 \le \frac{r}{2}n + \frac{r}{2}(n-1)n = \frac{r}{2}n^2. \end{equation}

Turning now to $L=-\L_{s,t}$, we have
\[ \frac{\del^2}{\del v_\ex\del v_\d} \mx{v}^\mx{k} = \begin{cases} \mx{k}(\ex)\mx{k}(\d) \mx{v}^{\mx{k}}/v_\ex v_\d, & \ex\ne\d, \\
\mx{k}(\ex)(\mx{k}(\ex)-1)\mx{v}^\mx{k}/v_\ex^2, & \ex=\d. \end{cases} \]
Thus, from (\ref{e.d.r}) we have
\[ \sum_{\ex,\d\in\EX}R^{\pm}_{\ex,\d}(\mx{v})\frac{\del^2}{\del v_\ex \del v_\d} \mx{v}^\mx{k} = \sum_{\ex\in\EX} \mx{k}(\ex)(\mx{k}(\ex)-1) \sum_{j,k=1}^{|\ex|} \pm v_{\ex^{(j)}\ex^{(k)}}\frac{\mx{v}^\mx{k}}{v_\ex^2} +\sum_{\ex\ne \d\in\EX} \mx{k}(\ex)\mx{k}(\d)\sum_{j=1}^{|\ex|}\sum_{k=1}^{|\d|} \pm v_{\ex^{(j)}\d^{(k)}}\frac{\mx{v}^\mx{k}}{v_\ex v_\d} \]
and so, from (\ref{e.d.rst}) and (\ref{e.d.DstLst}),
\begin{align*} \L_{s,t}(\mx{v}^\mx{k}) &= \frac12\sum_{\ex\in\EX} \mx{k}(\ex)(\mx{k}(\ex)-1) \sum_{j,k=1}^{|\ex|} \left[\left(s-\frac{t}{2}\right)(\pm) +\frac{t}{2}(\pm)\right] v_{\ex^{(j)}\ex^{(k)}}\frac{\mx{v}^\mx{k}}{v_\ex^2} \\
& \qquad + \frac12\sum_{\ex\ne\d\in\EX} \mx{k}(\ex)\mx{k}(\d)\sum_{j=1}^{|\ex|}\sum_{k=1}^{|\d|}  \left[\left(s-\frac{t}{2}\right)(\pm) +\frac{t}{2}(\pm)\right]  v_{\ex^{(j)}\d^{(k)}}\frac{\mx{v}^{\mx{k}}}{v_\ex v_\d}.
\end{align*}
As above, it follows that
\begin{align}\nonumber \|\L_{s,t}(\mx{v}^\mx{k})\|_1 &\le \frac{r}{2}\sum_{\ex\in\EX} \mx{k}(\ex)(\mx{k}(\ex)-1)\cdot|\ex|^2 + \frac{r}{2}\sum_{\ex\ne\d\in\EX} \mx{k}(\ex)\mx{k}(\d)|\ex||\d| \\ \nonumber
&\le \frac{r}{2}\sum_{\ex\in\EX} \mx{k}(\ex)^2|\ex|^2 + \frac{r}{2}\sum_{\ex\ne\d\in\EX} \mx{k}(\ex)\mx{k}(\d)|\ex||\d| \\ \label{e.Lst.est1}
&= \frac{r}{2}\sum_{\ex,\d\in\EX} \mx{k}(\ex)\mx{k}(\d)|\ex||\d| = \frac{r}{2}\left(\sum_{\ex\in\EX}\mx{k}(\ex)|\ex|\right)^2 \le \frac{r}{2}n^2.
\end{align}
As this holds for all $\mx{v}\in\B_n$, we therefore have
\begin{equation} \label{e.L.est1} \|L\|_{1\to1} = \max_{\mx{v}^\mx{k}\in\B_n} \|-\L_{s,t}(\mx{v}^\mx{k})\|_1 \le \frac{r}{2}n^2. \end{equation}
Combining (\ref{e.borg2}) with (\ref{e.D.est1}) and (\ref{e.L.est1}) proves the result. \end{proof}

When $s,t>0$ and $s>t/2$, $r = (s-\frac{t}{2})+\frac{t}{2} = s$.  Proposition \ref{prop Cnst bound} then shows that the constant $C(s,t,P)$ in Corollary \ref{cor.conc} can be bounded by
\begin{equation} \label{e.CstP1} C(s,t,P) \le \frac{s}{2}(\deg(P))^2 e^{s(\deg(P))^2}\|P\|_1, \qquad P\in\PP \end{equation}
by using $1/N^2\le 1$.  We can do better than this if we take $N$ sufficiently large.

\begin{corollary} \label{cor t.ind} Let $s,t\in\R$, and set $r=|s-\frac{t}{2}|+\frac12|t|$.  Let $\d>0$, $n,N\in\N$, and $P\in\PP_n$.  If $N > \sqrt{2/\d}$, then
\begin{equation} \label{e.t.ind} \left| \big(e^{-\D^N_{s,t}}P\big)(\mx{1}) - \big(e^{-\D_{s,t}}P\big)(\mx{1})\right| \le \frac{1}{N^2}\cdot\frac{1}{\d}e^{\frac{r}{2}(1+\d)n^2} \|P\|_1.
\end{equation}
\end{corollary}

\begin{proof} When $N>\sqrt{2/\d}$, we have $1+1/N^2 < 1+\d/2$, and so
\begin{equation} \label{e.wonk1} \frac{r}{2}n^2 e^{\frac{r}{2}n^2\left(1+\frac{1}{N^2}\right)} \le \frac{r}{2}n^2 e^{-\frac{r}{4}\d n^2} e^{\frac{r}{2}(1+\d)n^2}. \end{equation}
Elementary calculus shows that the function $x\mapsto xe^{-\d x/2}$ is maximized at $x=2/\d$, and takes value $2/e\d<1/\d$ there.  Substituting $x=\frac{r}{2}n^2$ in (\ref{e.wonk1}), the result now follows from (\ref{e.Cnst bound}).
\end{proof}

That being said, the author does not believe the estimate (\ref{e.t.ind}) on the constant $C(s,t,P)$ in (\ref{e.conc1}) is anywhere close to optimal: the above proofs involved fairly blunt estimates that ignored many potential cancellations. Indeed, if we work explicitly in the case $N=1$, for any linear polynomial $\H\PP\ni P = \sum_{k=-n}^n a_k v_k$, (\ref{e.Fourier1}) shows that
\[ \big(e^{-\D^1_{t,0}}P\big)(\mx{1}) = \left.\left(e^{\frac{t}{2}\Delta_{\U_1}} \left(\sum_{k=-n}^n a_k U^k\right)\right)\right|_{U=I_1} = \sum_{k=-n}^n a_k e^{-\frac{t}{2}k^2}\]
while (\ref{e.D-nu}) shows that
\[ \big(e^{-\D_{t,0}}P\big)(\mx{1}) = \sum_{k=-n}^n a_k\nu_k(t). \]
Thus, we have
\begin{equation} \label{e.uniform?} \left|\big(e^{-\D^1_{t,0}}P\big)(\mx{1})- \big(e^{-\D_{t,0}}P\big)(\mx{1})\right| \le \sum_{k=-n}^n |e^{-\frac{t}{2}k^2}-\nu_k(t)||a_k| \le 2\|P\|_1 \end{equation}
since $0<e^{-\frac{t}{2}k^2}\le 1$ and $|\nu_k(t)|\le 1$ (as it is a moment of a probability measure on $\U$).  On $\U_1$, every trace polynomial reduces to a polynomial in $U$ which intertwines with a linear polynomial (since $\tr (U^k) = U^k$ for $U\in\U_1$). This reduction process can only increase the $\|\cdot\|_1$-norm; cf.\ Remark \ref{r.basis}.  Thus, (\ref{e.uniform?}) shows that, in the special case $N=1$, there is a {\em uniform bound} (uniform in $n$ and $t$) for the concentration of expectations of polynomials in $\PP_n$.  It does not follow easily, unfortunately, that $C(s,t,P)$ is uniformly bounded in the $\U_N$ case; but the author strongly suspects this is so.  We leave the investigation of the precise behavior of the constants $C(s,t,P)$ to a future publication.

\section{Convergence of Empirical Distributions \label{section empirical convergence}}

This section is devoted to the proofs of Theorems \ref{thm 1}-\ref{thm 2}.  Theorem \ref{thm 1} is treated first, separately, with specialized techniques adapted from \cite{Levy2010}.   We then proceed with Theorem \ref{thm 2}, and then derive Theorems \ref{thm 1.1} and \ref{thm 1.2} essentially as special cases.

\subsection{Empirical Eigenvalues on $\U_N$\label{section proof of thm 1}}

Let $f\colon\U\to\C$ be a measurable function.  Since the group $\U_N$ consists of normal matrices, functional calculus is available to us.  From (\ref{e.emp1.5}), the empirical integral $\int_\U f\,d\widetilde{\nu}^N_t$ is the random variable
\begin{equation} \label{e.intU0} \int_\U f\,d\widetilde{\nu}^N_t = \tr\circ f_N \quad \text{on} \quad (\U_N,\rho^N_t). \end{equation}
We will initially bound the empirical integral in terms of the {\em Lipschitz norm} on test functions.  A function $F\colon\U_N\to\C$ is {\em Lipschitz} if
\[ \|F\|_{\Lip(\U_N)} \equiv \sup_{U\ne V\in\U_N} \frac{|F(U)-F(V)|}{d_{\U_N}(U,V)}<\infty, \]
where $d_{\U_N}$ is the Riemannian distance on $\U_N$ given by the Riemannian metric induced by the inner product (\ref{e.innprod2}) on $\u_N$.  In the special case $N=1$, this is just arclength distance:
\begin{equation} \label{e.LipU1} \|f\|_{\Lip(\U)} = \sup_{\alpha\ne\beta\in[0,2\pi)}\frac{|f(e^{i\alpha}-e^{i\beta})|}{|\alpha-\beta|}. \end{equation}
The following general lemma was given in \cite[Proposition 4.1]{Levy2010}; it is adapted from the now well-known techniques in \cite{Guionnet2010}, and attributable to earlier work of Talagrand.
\begin{lemma}[L\'evy, Ma\"ida, 2010] \label{l.Levy4.1} Let $N\in\N$.  If $f\colon\U\to\C$ is Lipschitz, then $\tr\circ f_N\colon\U_N\to\C$ is Lipschitz, and
\begin{equation} \label{e.Levy4.1} \|\tr\circ f_N\|_{\Lip(\U_N)} = \frac{1}{N}\|f\|_{\Lip(\U)}. \end{equation}
\end{lemma}
\begin{remark} Lemma \ref{l.Levy4.1} is proved in \cite{Levy2010} only for {\em real}-valued $f$; but the proof works without modification for complex valued test functions. \end{remark}

Lemma \ref{l.Levy4.1} is then used in conjunction with the following, proved as \cite[Proposition 6.1]{Levy2010}.

\begin{lemma}[L\'evy, Ma\"ida, 2010] \label{l.Levy6.1} Let $F\colon\U_N\to\R$ be Lipschitz, and let $N\in\N$.  Then for $t\ge 0$,
\begin{equation} \label{e.Var.Levy} \Var_{\rho^N_t}(F) \le t\|F\|_{\Lip(\U_N)}^2. \end{equation}
\end{lemma}
\noindent Lemma \ref{l.Levy6.1} is proved using a fairly well-known martingale method.  If $U^N_t$ is a Brownian motion on $\U_N$ (i.e.\ the Markov process with generator $\frac12\Delta_{\U_N}$), and $T>0$, then for any $L^2$-function $F\colon\U_N\to\R$, the real-valued stochastic process
\[ t\mapsto \big(e^{\frac{1}{2}(T-t)\Delta_{\U_N}}F\big)(U^N_t) \]
is a martingale, which is well-behaved when $F$ is Lipschitz (in particular since $\|e^{\frac{t}{2}\Delta_{\U_N}}F\|_{\Lip(\U_N)} \le \|F\|_{\Lip(\U_N)}$ for any $t\ge 0$).
Our first task is to generalize Lemma \ref{l.Levy6.1} in two ways: from variances to covariances, and from real-valued to complex-valued random variables.
\begin{corollary} \label{cor.Levy6.1} Let $N\in\N$ and $t\ge 0$.  If $F,G\colon\U_N\to\C$ are Lipschitz functions, then
\begin{equation} \label{e.Cov.Levy} \left|\Cov_{\rho^N_t}(F,G)\right| \le 2t\|F\|_{\Lip(\U_N)}\|G\|_{\Lip(\U_N)}. \end{equation}
\end{corollary}

\begin{remark} To be clear: for two complex-valued $L^2$ random variables $F$ and $G$, $\Cov(F,G) = \E(F\overline{G}) - \E(F)\E(\overline{G}) = \E[(F-\E(F))(\overline{G}-\E(\overline{G}))]$. \end{remark}

\begin{proof} From the Cauchy-Schwarz inequality, we have
\begin{equation} \label{e.Cov.lem1} \left|\Cov(F,G)\right| = \left|\E\big[(F-\E(F))(\overline{G}-\E(\overline{G}))\big]\right| \le \|F-\E(F)\|_{L^2}\|G-\E(G)\|_{L^2} = \sqrt{\Var(F)\Var(G)}. \end{equation}
Note that, for a complex-valued random variable $F = F_1+iF_2$, $\Var(F_1+iF_2) = \Var(F_1) + \Var(F_2)$.  A complex-valued function is Lipschitz iff its real and imaginary parts are both Lipschitz, and so Lemma \ref{l.Levy6.1} shows that
\begin{equation} \label{e.Cov.lem2} \Var_{\rho^N_t}(F_1+iF_2) = \Var_{\rho^N_t}(F_1) + \Var_{\rho^N_t}(F_2) \le t\left(\|F_1\|_{\Lip(\U_N)}^2 + \|F_2\|_{\Lip(\U_N)}^2\right). \end{equation}
We now estimate
\begin{align*} \|F_1\|_{\Lip(\U_N)}^2 + \|F_2\|_{\Lip(\U_N)}^2 &\le 2\max\left\{\|F_1\|_{\Lip(\U_N)}^2,\|F_2\|_{\Lip(\U_N)}^2\right\} \\
&\le 2\max\left\{\sup_{U\ne V}\frac{(F_1(U)-F_1(V))^2}{d_{\U_N}(U,V)^2},\sup_{U\ne V}\frac{(F_2(U)-F_2(V))^2}{d_{\U_N}(U,V)^2}\right\} \\
&\le 2\sup_{U\ne V}\left[\frac{(F_1(U)-F_1(V))^2}{d_{\U_N}(U,V)^2} + \frac{(F_2(U)-F_2(V))^2}{d_{\U_N}(U,V)^2}\right] \\
&= 2\|F_1+iF_2\|_{\Lip(\U_N)}^2,
\end{align*}
where the penultimate inequality is just the statement that if $f_1,f_2\ge 0$ then $\sup(f_1+f_2) \ge \max\{\sup f_1,\sup f_2\}$.  Combining this with (\ref{e.Cov.lem1}) and (\ref{e.Cov.lem2}) proves the (\ref{e.Cov.Levy}).  \end{proof}

\begin{remark} It is likely that the variance estimate (\ref{e.Var.Levy}) holds as stated for complex-valued $F$, but this is not immediately clear from the proof as given.  Since we do not care too much about exact constants, we are content to have a possibly-extraneous factor of $2$ in (\ref{e.Cov.Levy}). \end{remark}

Combining Lemma \ref{l.Levy4.1} and Corollary \ref{cor.Levy6.1} (in the special case $F=G$) with (\ref{e.intU0}) immediately proves (\ref{e.conv3}) in Theorem \ref{thm 1}.  We will now show that, at the expense of decreasing the speed of convergence below $O(1/N^2$) (but still summably fast), convergence holds for the much less regular functions in the Sobolev spaces $H_p(\U)$ for $p>1$.  (If $p<\frac32$, $H_p(\U)$ consists primarily of non-Lipschitz functions; cf.\ Section \ref{section Gevrey}.)  We begin by considering trigonometric polynomial test functions.

\begin{proposition} \label{p.Sobolev1} Let $n\in\N$, and let $f(u) = \sum_{k=-n}^n \hat{f}(k)u^k$ be a trigonometric polynomial on $\U$.  If $\frac12<p<\frac32$, then
\begin{equation} \Var\left(\int_\U f\,d\widetilde{\nu}^N_t\right) \le \frac{n^{3-2p}}{N^2}\cdot\frac{8t}{3-2p}\|f\|_{H_p(\U)}^2. \end{equation}
\end{proposition}

\begin{proof} From (\ref{e.emp1.5}), the empirical integral is the random variable
\[ \int_\U f\,d\widetilde{\nu}^N_t = \sum_{k=-n}^n \hat{f}(k) \tr[(\cdot)^k], \]
and so we can expand the variance as
\begin{equation} \label{e.trigpoly1} \Var\left(\int_\U f\,d\widetilde{\nu}^N_t\right) = \sum_{|j|,|k|\le n} \hat{f}(j)\overline{\hat{f}(k)} \Cov_{\rho^N_t}\big(\tr[(\cdot)^j],\tr[(\cdot)^k]\big). \end{equation}
Using Corollary \ref{cor.Levy6.1} and then Lemma \ref{l.Levy4.1}, we have
\begin{equation} \label{e.trigpoly2} \left|\Cov_{\rho^N_t}\big(\tr[(\cdot)^j],\tr[(\cdot)^k]\big)\right| \le 2t\|\tr[(\cdot)^j]\|_{\Lip(\U_N)}\|\tr[(\cdot)^k]\|_{\Lip(\U_N)} =\frac{2t}{N^2}\|\chi_j\|_{\Lip(\U)}\|\chi_k\|_{\Lip(\U)}, \end{equation}
where $\chi_k(u) = u^k$; cf.\ Section \ref{section Gevrey}. Since the functions $\chi_k$ are in $C^1(\U)$, we can compute their Lipschitz norms as
\[ \|\chi_k\|_{\Lip(\U)} = \sup_\U |\chi_k'| = |k|. \]
Combining this with (\ref{e.trigpoly1}) and (\ref{e.trigpoly2}) yields
\begin{equation} \label{e.trigpoly3} \Var\left(\int_\U f\,d\widetilde{\nu}^N_t\right) \le \frac{2t}{N^2}\sum_{|j|,|k|\le n} |\hat{f}(j)||\hat{f}(k)||j||k| = \frac{2t}{N^2}\left(\sum_{k=-n}^n |k||\hat{f}(k)|\right)^2. \end{equation}
Note that the $k=0$ term in the squared-sum is $0$, so we omit it from here on. We estimate this squared-sum with the Cauchy-Schwarz inequality, applied with $|k| = |k|^{1-p}|k|^p$:
\begin{align} \nonumber \left(\sum_{1\le|k|\le n} |k||\hat{f}(k)|\right)^2 &\le \left(\sum_{1\le|k|\le n} |k|^{2(1-p)}\right)\cdot\left(\sum_{1\le|k|\le n} |k|^{2p}|\hat{f}(k)|^2\right) \\  \label{e.trigpoly4}
&\le \left(\sum_{1\le|k|\le n} |k|^{2(1-p)}\right)\cdot \|f\|_{H_p(\U)}^2, \end{align}
where the Sobolev $H_p$-norm is defined in (\ref{e.d.Hp}).  Let $r=2(p-1)$; then $0<r\le 1$.  We utilize the calculus estimate
\[ \sum_{k=1}^\infty\frac{1}{k^r} \le 2^r\int_1^{n+1} \frac{dx}{x^r} = \frac{2^r}{1-r}[(n+1)^{1-r}-1] \le \frac{2}{1-r}n^{1-r}, \]
which yields
\begin{equation} \label{e.trigpoly.calc2} \sum_{1\le|k|\le n}|k|^{2(1-p)} = 2\sum_{k=1}^\infty k^{2(1-p)} \le \frac{4}{3-2p}n^{3-2p}. \end{equation}
Equations (\ref{e.trigpoly4}) and (\ref{e.trigpoly.calc2}) prove the proposition.
\end{proof}

\begin{remark} In the regime $p>\frac32$, where $2(1-p)<-1$, the sum in (\ref{e.trigpoly4}) is uniformly bounded in $n$, and the resulting estimate on the variance is
\[  \Var\left(\int_\U f\,d\widetilde{\nu}^N_t\right) \le \frac{1}{N^2}\cdot \frac{4^pt}{2p-3}\|f\|_{H^p(\U)}^2, \quad p>\frac32. \]
In the case $p=\frac32$, $H_p(\U)$ corresponds roughly with Lipschitz functions, and so (\ref{e.conv3}) is the optimal result.
\end{remark}

We will use Proposition \ref{p.Sobolev1} to prove (\ref{e.conv2}) by doing a band-limit cut-off of the test function $f$ at a frequency $n$ that grows with $N$ (in fact, the optimal result is achieved at $n=N$).  To proceed, we first need the following lemma.

\begin{lemma} \label{l.thm1} Let $N\in\N$ and $t\ge 0$.  For $f\in L^\infty(\U)$,
\begin{equation} \label{e.l.thm1} \Var\left(\int_\U f\,d\widetilde{\nu}^N_t\right) \le 4\|f\|_{L^\infty(\U)}^2. \end{equation}
\end{lemma}

\begin{proof} For any $L^2$ random variable $F$, we utilize the crude estimate
\[ \Var(F) = \|F-\E(F)\|_{L^2}^2 \le \left(\|F\|_{L^2} + |\E(F)|\right)^2 \le 4\|F\|_{L^2}^2. \]
With $F = \int_\U f\,d\widetilde{\nu}^N_t$, (\ref{e.emp1}) shows that, for $U\in\U_N$,
\begin{equation*} |F(U)| = \frac{1}{N}\Big|\sum_{\lambda\in\Lambda(U)} f(\lambda)\Big| \le \|f\|_{L^\infty(\U)} \end{equation*}
since $\Lambda(U)$ is a set of size $N$.  Since $\rho^N_t$ is a probability measure, it follows that $\|F\|_{L^2(\rho^N_t)} \le \|f\|_{L^\infty(\U)}$, and the result follow.  \end{proof}

We now proceed to prove (\ref{e.conv2}) in Theorem \ref{thm 1}.

\begin{proposition} \label{p.thm1-2} Let $t\ge 0$, $N\in\N$ and $1<p<\frac32$.  For $f\in H_p(\U)$,
\begin{equation} \label{e.p.thm1-2} \Var\left(\int_\U f\,d\widetilde{\nu}^N_t\right)  \le \frac{1}{N^{2p-1}}\cdot 8\|f\|_{H_p(\U)}\left(\frac{\sqrt{t}}{\sqrt{3-2p}}+\frac{1}{\sqrt{2p-1}}\right)^2. \end{equation}
\end{proposition}

\begin{proof} Fix $f\in H_p(\U)$, with Fourier expansion $f=\sum_{k\in\Z} \hat{f}(k)\chi_k$.  Let
\[ f_N = \sum_{k=-N}^N \hat{f}(k)\chi_k \]
be the band-limited frequency cut-off at level $N$, and define
 \[ F_N = \int_\U f_N\,d\widetilde{\nu}^N_t, \quad \text{and} \quad F^N =  \int_\U (f-f_N)\,d\widetilde{\nu}^N_t, \]
 so that $F_N+F^N = \int_\U f\,d\widetilde{\nu}^N_t$.  From the triangle inequality for  $L^2$,
\begin{equation} \label{e.p.thm1-2.1} \left(\Var\left(\int_\U f\,d\widetilde{\nu}^N_t\right)\right)^{1/2} = \sqrt{\Var(F_N+F^N)} \le \sqrt{\Var(F_N^{})} + \sqrt{\Var(F^N)}. \end{equation}
From Proposition \ref{p.Sobolev1}, the square of the first term in (\ref{e.p.thm1-2.1}) is bounded by
\begin{equation} \label{e.p.thm1-2.2} \Var(F_N) \le \frac{N^{3-2p}}{N^2}\cdot \frac{8t}{3-2p}\|f_N\|_{H_p(\U)}^2 \le  N^{1-2p}\cdot\frac{8t}{3-2p}\|f\|_{H^p(\U)}^2. \end{equation}
From Lemma \ref{l.thm1}, the square of the second term in (\ref{e.p.thm1-2.1}) is bounded by
\begin{equation} \label{e.p.thm1-2.3} \Var(F^N) \le 4\|f-f_N\|_{L^\infty(\U)}^2, \end{equation}
which we can bound as follows:
\begin{align} \nonumber \sup_{u\in\U} |f(u)-f_N(u)|^2 = \sup_{u\in\U}\left|\sum_{|k|>N} \hat{f}(k)u^k\right|^2 \le \left(\sum_{|k|>N} |\hat{f}(k)|\right)^2 &\le \left(\sum_{|k|>N} |k|^{-2p}\right)\cdot\left(\sum_{|k|>N} |k|^{2p}|\hat{f}(k)|^2\right) \\ \label{e.thm1-2.sum}
&\le \left(\sum_{|k|>N} |k|^{-2p}\right)\|f\|_{H^p(\U)}^2. \end{align}
We can bound the above sum as in (\ref{e.trigpoly.calc2}), using the calculus estimate
\[ \sum_{k=N+1}^\infty \frac{1}{k^{2p}} \le \int_N^\infty \frac{dx}{x^{2p}} = \frac{1}{2p-1}N^{1-2p}. \]
Combining this with (\ref{e.p.thm1-2.3}) and (\ref{e.thm1-2.sum}) yields
\begin{equation} \label{e.thm1-2.final} \Var(F^N) \le N^{1-2p}\cdot \frac{8}{2p-1}\|f\|_{H^p(\U)}^2. \end{equation}
Combining (\ref{e.p.thm1-2.1}), (\ref{e.p.thm1-2.2}), (\ref{e.thm1-2.final}) proves (\ref{e.p.thm1-2}).  \end{proof}

This brings us to the proof of Theorem \ref{thm 1}.

\begin{proof}[Proof of Theorem \ref{thm 1}] \label{proof of thm 1} Proposition \ref{p.thm1-2} proves (\ref{e.conv2}), while, as remarked above, Lemma \ref{l.Levy4.1} and Corollary \ref{cor.Levy6.1} prove (\ref{e.conv3}).  Thus, we are left to prove only (\ref{e.conv1}).   Fix $f\in C(\U)$, and let $\e>0$.  By the Weierstrass approximation theorem, there is a trigonometric polynomial $g_\e$ on $\U$ such that $\|f-g_\e\|_{L^\infty(\U)} < \sqrt{\e}/4$.  Let
\[ F=\int_\U f\,d\widetilde{\nu}^N_t, \quad \text{and} \quad G = \int_\U g_\e\,d\widetilde{\nu}^N_t. \]
Then, as in (\ref{e.p.thm1-2.1}),we estimate
\begin{equation} \label{e.thm1.final1} \sqrt{\Var(F)} \le \sqrt{\Var(G)} + \sqrt{\Var(F-G)} \le \sqrt{\Var(G)} + 2\|f-g_\e\|_{\infty} <\sqrt{\Var(G)}+\sqrt{\e}/2 \end{equation}
by Lemma \ref{l.thm1}.  Now, $g_\e$ is Lipschitz, and so (\ref{e.conv3}) gives
\begin{equation} \label{e.thm1.final2} \sqrt{\Var(G)} \le \frac{2t}{N^2}\|g_\e\|_{\Lip(\U)}. \end{equation}
Thus, for any $N> 2\sqrt{t\|g_\e\|_{\Lip(\U)}}/\e^{1/4}$, $\sqrt{\Var(G)} < \sqrt{\e}/2$, and so (\ref{e.thm1.final1}) and (\ref{e.thm1.final2}) show that $\Var(F) = \Var(\int_\U f\,d\widetilde{\nu}^N_t)<\e$ for all sufficiently large $N$. Convergence in probability (\ref{e.conv1}) now follows immediately from Chebyshev's inequality.  \end{proof}

For a discussion of the (lack of) sharpness of (\ref{e.conv2}), see the end of Section \ref{section proof of thm 2}.

\subsection{Empirical Noncommutative Distribution on $\GL_N$\label{section proof of thm 2}}

\begin{definition} \label{d.phist} Let $s,t\in\R$, and let $\D_{s,t}$ be the intertwining operator on $\PP$ given in Theorem \ref{thm intertwines}.  For each $n$, the finite-dimensional subspace $\PP_n$ is invariant under $\D_{s,t}$, and so $e^{-\D_{s,t}}\colon \PP\to\PP$ is well-defined.  Define the {\bf noncommutative distribution} $\p_{s,t}\colon \C\langle A,A^\ast\rangle \to \C$ to be the following linear functional:
\begin{equation} \label{e.phist} \p_{s,t}(f) = \big(e^{-\D_{s,t}}\Upsilon(f)\big)(\mx{1}), \qquad f\in\C\langle A,A^\ast\rangle \end{equation}
where $\Upsilon\colon\C\langle A,A^\ast\rangle\hookrightarrow\PP^+$ is the inclusion of (\ref{e.Upsilon}).
\end{definition}
\noindent To be clear: $\D_{s,t}$ does not preserve the space $\Upsilon(\C\langle A,A^\ast\rangle)$ of linear polynomials, and so $e^{-\D_{s,t}}f$ contains terms of higher (ordinary) degree, although it preserves the {\em trace} degree of $\Upsilon(f)$.  The functional $\p_{s,t}$ is defined by evaluating the resultant polynomial function $\mx{v}\mapsto \big(e^{-\D_{s,t}}\Upsilon(f)\big)(\mx{v}) \in \PP^+$ at $\mx{v}=\mx{1}$.

\begin{remark} \label{r.nothomom} It is tempting to think that $\p_{s,t}$ is therefore a homomorphism on $\C\langle A,A^\ast\rangle$, since $e^{-\D_{s,t}}$ is a homomorphism on $\PP^+$.  However, $\Upsilon$ is {\em not} a homomorphism.  The product on $\C\langle A,A^{\ast}\rangle$ is incompatible with the product on the larger space $\PP^+$; it is the difference between convolution product and pointwise product of functions.  \end{remark}

To properly call the linear functional (\ref{e.phist}) a non-commutative distribution, we must realize it as the distribution of a random variable in a noncommutative probability space $(\A_{s,t},\t_{s,t})$.  This is done in precisely the same way that we constructed the mean $\E(\widetilde{\p}^N)$ of an empirical distribution (\ref{e.emp4}) as a genuine noncommutative distribution.  We take $\A_{s,t} = \C\langle A,A^\ast\rangle$, and define $\t_{s,t}(f) = \p_{s,t}(f)$ for $f\in\A$; then $\p_{s,t} = \p_a$ where $a\in\C\langle A,A^\ast\rangle$ is the coordinate random variable $a(A,A^\ast) = A$.  Note that $\p_{s,t}(1)=1$ since $\D_{s,t}$ annihilates constants.  That $\t_{s,t}$ is tracial and positive semi-definite actually follows from Theorem \ref{thm 2}: (\ref{e.ncdist1}) identifies $\p_{s,t}$ as the limit of the mean distributions $\E(\widetilde{\p}^N_{s,t})$ which are tracial and positive definite (since $\mu^N_{s,t}$ has infinite support); see the discussion on page \pageref{e.emp4}.  It is straightforward to verify that a limit of tracial states is tracial, and hence $\t_{s,t}$ is a tracial state.  What is not  so clear is whether $\t_{s,t}$ is {\em faithful}, as this property does not generally survive under limits.  In the special case $s=t$, the concurrent paper \cite{Guillaume2013} proves that $\p_{t,t}$ is the noncommutative distribution of the free multiplicative Brownian motion $z_t$ of (\ref{e.fmbmsde}), and so in this case, $\t_{t,t}$ is known to be faithful.  We leave the general question of faithfulness of $\t_{s,t}$, and other noncommutative probabilistic questions, to future consideration.

The key to proving Theorem \ref{thm 2} is the following extension of Corollary \ref{cor.conc}.  We will use it here only in the diagonal case ($P=Q$), but the general covariance estimate will be useful in Sections \ref{section empirical GL_N} and \ref{section empirical MN>0}.

\begin{proposition} \label{p.thm2.0} For $P,Q\in\PP$, there is a constant $C_2(s,t,P,Q)$ depending continuously on $s,t$ so that, for each $N\in\N$,
\begin{equation} \label{e.thm2.0} \left|\Cov_{\mu^N_{s,t}}\big(P\circ\mx{V}_N,Q\circ\mx{V}_N\big)\right| \le \frac{1}{N^2}\cdot C_2(s,t,P,Q). \end{equation}
\end{proposition}

\begin{proof} From (\ref{e.intertwineC}), we may write
\[ P\circ\mx{V}_N\cdot \overline{Q\circ\mx{V}_N} = \big(PQ^\ast\big)\circ\mx{V}_N\]
where $Q^\ast = \CC Q$.  Thus, (\ref{e.hkm2'}) shows that
\begin{equation} \label{e.ncvar1}
\E_{\mu^N_{s,t}}\left(P\circ\mx{V}_N\cdot\overline{Q\circ\mx{V}_N}\right) = \big(e^{-\D^N_{s,t}}(PQ^\ast)\big)(\mx{1}).
\end{equation}
Similarly,
\begin{equation} \label{e.ncvar2}
\E_{\mu^N_{s,t}}(P\circ\mx{V}_N)\cdot \E_{\mu^N_{s,t}}(\overline{Q\circ\mx{V}_N})= \big(e^{-\D^N_{s,t}}P\big)(\mx{1})\cdot \big(e^{-\D^N_{s,t}}Q^\ast\big)(\mx{1}).
\end{equation}
To simplify notation, we suppress $s,t$ and denote
\begin{align} \label{e.Psi1} \Psi^N_1 \equiv \big(e^{-\D^N_{s,t}}P\big)(\mx{1}), \quad \Psi^N_\ast \equiv \big(e^{-\D^N_{s,t}}Q^\ast\big)(\mx{1}), \quad \Psi^N_{1,\ast} \equiv \big(e^{-\D^N_{s,t}}(PQ^\ast)\big)(\mx{1}), \\ \label{e.Psi2}
\Psi_1 \equiv \big(e^{-\D_{s,t}}P\big)(\mx{1}), \quad \Psi_\ast \equiv \big(e^{-\D_{s,t}}Q^\ast\big)(\mx{1}), \quad \Psi_{1,\ast} \equiv \big(e^{-\D_{s,t}}(PQ^\ast)\big)(\mx{1}). \end{align}
Thus, (\ref{e.ncvar1}) and (\ref{e.ncvar2}) show that
\begin{equation} \label{e.ncvar3} \Cov_{\mu^N_{s,t}}(P\circ\mx{V}_N,Q\circ\mx{V}_N)= \Psi^N_{1,\ast}-\Psi^N_1\Psi^N_\ast. \end{equation}
We estimate this as follows.  First
\begin{equation} \label{e.ncvar4} |\Psi^N_{1,\ast}-\Psi^N_1\Psi^N_\ast| \le |\Psi^N_{1,\ast}-\Psi_{1,\ast}| + |\Psi_{1,\ast}-\Psi_1\Psi_\ast| + |\Psi_1\Psi_\ast - \Psi_1^N\Psi^N_\ast|. \end{equation}
Referring to (\ref{e.Psi2}), since $e^{-\D_{s,t}}$ is a homomorphism, the second term in (\ref{e.ncvar4}) is $0$.  The first term is bounded by $\frac{1}{N^2}\cdot C(s,t,PQ^\ast)$ by Corollary \ref{cor.conc}.  For the third term, we add and subtract $\Psi_1^N\Psi_\ast$ to make the additional estimate
\begin{align} \nonumber |\Psi_1\Psi_\ast - \Psi_1^N\Psi_\ast^N| &\le |\Psi_\ast||\Psi_1-\Psi_1^N| + |\Psi_1^N||\Psi_\ast-\Psi_\ast^N| \\ \nonumber
&\le |\Psi_\ast||\Psi_1-\Psi_1^N|  + \big(|\Psi_1| + |\Psi_1^N-\Psi_1|)|\Psi_\ast-\Psi^N_\ast| \\ \nonumber
&\le \frac{1}{N^2}\cdot |\Psi_\ast|C(s,t,P) + \left(|\Psi_1| + \frac{1}{N^2}\cdot C(s,t,P)\right)\cdot\frac{1}{N^2}\cdot C(s,t,Q^\ast) \\ \label{e.ncvar5}
&= \frac{1}{N^2}\cdot\left(|\Psi_\ast|C(s,t,P)+|\Psi_1|C(s,t,Q^\ast)\right) + \frac{1}{N^4}\cdot C(s,t,P)C(s,t,Q^\ast).
\end{align}
Combining (\ref{e.ncvar5}) with (\ref{e.ncvar3}) -- (\ref{e.ncvar4}) and the following discussion shows that the constant
\begin{equation} \label{e.C_2} C_2(s,t,P,Q) = C(s,t,PQ^\ast) + C(s,t,P)C(s,t,Q^\ast) + |\Psi_\ast|C(s,t,P)+|\Psi_1|C(s,t,Q^\ast) \end{equation}
verifies (\ref{p.thm2.0}), proving the proposition.  \end{proof}

Proposition \ref{p.thm2.0} shows that any trace polynomial in $Z^N_{s,t}$ has variance of order $1/N^2$, as discussed following the statement of Theorem \ref{thm 2}.  The theorem follows as a very special case, due to the following.

\begin{lemma} \label{l.UpsilonV} Let $Z\in\GL_N$, and let $f\in\C\langle A,A^\ast\rangle$.  Let $\p_Z$ denote the noncommutative distribution of $Z$ with respect to $(\M_N,\tr)$ (Definition \ref{d.ncdist0}), let $\Upsilon\colon\C\langle A,A^\ast\rangle\hookrightarrow\PP^+$ be the map of (\ref{e.Upsilon}), and let $\mx{V}_N$ be the map in Notation \ref{n.poly3}. Then
\[ \p_Z(f) = (\Upsilon(f)\circ\mx{V}_N)(Z). \]
\end{lemma}

\begin{proof} As both sides are linear functions of $f$, it suffices to prove the claim on basis elements $f(A,A^\ast) = A^\ex$ for some $\ex\in\EX^+$.  Then $\Upsilon(f) = v_\ex$, and $(v_\ex)\circ\mx{V}_N(Z) = \tr(Z^\ex) = \p_Z(A^\ex)$ as claimed.
\end{proof}

This brings us to the proof of Theorem \ref{thm 2}.

\begin{proof}[Proof of Theorem \ref{thm 2}] \label{proof of thm 2}  We begin by establishing that (\ref{e.ncdist1}) holds with the linear functional $\p_{s,t}$ of Definition \ref{d.phist}.  From (\ref{e.emp4}), we have
\[ \E(\widetilde{\p}_{s,t}^N)(f) = \int_{\GL_N} \p_Z(f)\,\mu^N_{s,t}(dZ) \]
where $\p_Z$ is the noncommutative distribution of $Z$ in $(\M_N,\tr)$.  Applying Lemma \ref{l.UpsilonV} and (\ref{e.hkm2'}) yields
\begin{equation} \label{e.thm2.1} \E(\widetilde{\p}_{s,t}^N)(f) = \E_{\mu_{s,t}^N} (\Upsilon(f)\circ\mx{V}_N) = \big(e^{-\D_{s,t}^N}\Upsilon(f)\big)(\mx{1}). \end{equation}
From the definition (\ref{e.phist}) of the limit distribution $\p_{s,t}$, (\ref{e.thm2.1}) shows that
\[ \left|\E(\widetilde{\p}_{s,t}^N)(f) - \p_{s,t}(f)\right| = \left|\big(e^{-\D_{s,t}^N}\Upsilon(f)\big)(\mx{1})-\big(e^{-\D_{s,t}}\Upsilon(f)\big)(\mx{1})\right| \le \frac{1}{N^2}\cdot C(s,t,\Upsilon(f)) \]
by Corollary \ref{cor.conc}; this proves (\ref{e.ncdist1}).

The random variable $\widetilde{\p}^N_{s,t}$ on the probability space $(\GL_N,\mu^N_{s,t})$ has value $\p_Z$ at $Z\in\GL_N$.  Thus, using Lemma \ref{l.UpsilonV}, we have
\begin{equation} \label{e.uselater0} \Var[\widetilde{\p}^N_{s,t}(f)] = \Var_{\mu_{s,t}^N}(\Upsilon(f)\circ\mx{V}_N), \end{equation}
and (\ref{e.ncdist2}) follows immediately from Proposition \ref{p.thm2.0}.  \end{proof}

We now give some quantitative estimate for the constant $C_2(s,t,P,Q)$ of (\ref{e.C_2}).  First we need to bound the terms $|\Psi_\ast|$ and $|\Psi_1|$ in that equation.

\begin{lemma} \label{l.Psi*} Let $s,t\in\R$, with $r=|s-\frac{t}{2}|+\frac12|t|$.  Let $n,N\in\N$, and let $P\in\PP_n$.  Then
\begin{equation} \label{e.Psi*} \left|\big(e^{-\D^N_{s,t}}P\big)(\mx{1})\right|\le e^{\frac{r}{2}(1+1/N^2)n^2}\|P\|_1, \qquad \text{and} \qquad  \left|\big(e^{-\D_{s,t}}P\big)(\mx{1})\right|\le e^{\frac{r}{2}n^2}\|P\|_1. \end{equation}
\end{lemma}

\begin{proof} Following (\ref{e.psi0}), (\ref{e.D.est1}), and (\ref{e.L.est1}), we estimate
\[ \left|\big(e^{-\D^N_{s,t}}P\big)(\mx{1})\right|\le \|e^{-\D^N_{s,t}}P\|_1 \le \|e^{-\left.\D^N_{s,t}\right|_{\PP_n}}\|_{1\to1}\|P\|_1 \le e^{\|\left.\D^N_{s,t}\right|_{\PP_n}\|_{1\to1}}\|P\|_1 \le e^{\frac{r}{2}(1+1/N^2)n^2}\|P\|_1, \]
proving the first inequality in (\ref{e.Psi*}).  The second follows by taking $N\to\infty$. \end{proof}

\begin{corollary} \label{c.C_2} Let $s,t\in\R$ with $r=|s-\frac{t}{2}|+\frac12|t|$, $n,m,N\in\N$, and $0<\d<1$. For $P\in\PP_n$, $Q\in\PP_m$, and $N>\sqrt{2/\d}$,
\begin{equation} \label{e.C_2.2} \Cov_{\mu^N_{s,t}}\big(P\circ\mx{V}_N,Q\circ\mx{V}_N\big) \le \frac{1}{N^2}\cdot \frac{4}{\d^2} e^{r(1+\d)(n^2+m^2)}\|P\|_1\|Q\|_1. \end{equation}
\end{corollary}

\begin{proof} The polynomial $Q^\ast$ has trace degree $m$, and so $PQ^\ast$ has trace degree $n+m$.  It therefore follows from (\ref{e.C_2}), together with Corollary \ref{cor t.ind} and Lemma \ref{l.Psi*}, that
\begin{align*} &\Cov_{\mu^N_{s,t}}\big(P\circ\mx{V}_N,Q\circ\mx{V}_N\big) \le \frac{1}{N^2}\cdot\Big[ \frac{1}{\d}e^{\frac{r}{2}(1+\d)(n+m)^2}\|PQ^\ast\|_1 \\
&\qquad\qquad\qquad\qquad + \frac{1}{\d^2}e^{\frac{r}{2}(1+\d)(n^2+m^2)}\|P\|_1\|Q^\ast\|_1 + \frac{1}{\d}\big(e^{\frac{r}{2}m^2}e^{\frac{r}{2}(1+\d)n^2}+e^{\frac{r}{2}n^2}e^{\frac{r}{2}(1+\d)m^2}\big)\|P\|_1\|P^\ast\|_1\Big]. \end{align*}
The reader can readily verify that $\|P^\ast\|_1 = \|P\|_1$ and $\|PQ^\ast\|_1 \le \|P\|_1\|Q\|_1$.  Together with the estimate $(n+m)^2 \le 2(n^2+m^2)$ and blunt bounds proves (\ref{e.C_2.2}).  \end{proof}

We conclude this section with a brief discussion of what bounds we expect are sharp, and the consequences this would have for the proof of Theorem \ref{thm 1}.  As mentioned in the remarks following the suggestive calculation (\ref{e.uniform?}), it is possible that the constants $C(t,0,P)$ of Corollary \ref{e.conc1} are uniformly bounded over $P\in\H\PP$.  To be precise, we conjecture that there is a constant $C(t)$ (depending continuously on $t>0$) so that
\begin{equation} \label{e.conc2?} \left|\big(e^{-\D^N_{t,0}}P\big)(\mx{1}) - \big(e^{-\D_{t,0}}P\big)(\mx{1})\right| \le \frac{C(t)}{N^2}, \qquad P\in\H\PP. \end{equation}
If this holds true, then as in the proof of Proposition \ref{p.thm2.0}, it would follow that there is a constant $C_2(t)$ such that, for $j,k\in\Z$,
\begin{equation} \label{e.conc3?}  \left|\Cov_{\rho^N_{s,t}}\big(v_j\circ\mx{V}_N,v_k\circ\mx{V}_N\big)\right| \le \frac{C_2(t)}{N^2}. \end{equation}
Indeed: the terms $|\Psi|$ and $|\Psi_\ast|$ in (\ref{e.Psi1}) and (\ref{e.C_2}) are $\le 1$, since $\Psi_1 = \lim_{N\to\infty}\Psi_1^N = \lim_{N\to\infty}\E_{\rho^N_t}\tr[(\cdot)^j]$ and $|\tr(U^j)| \le 1$ for $U\in\U_N$ (see the proof of Lemma \ref{l.thm1}), and similarly $|\Psi_\ast| = \lim_{N\to\infty}|\tr[(\cdot)^{-k}]|\le 1$.

Consider, then, $f\in H_p(\U)$ with $p>\frac12$; the covariance expansion (\ref{e.trigpoly1}) together with the conjectured (\ref{e.conc3?})  yields
\[ \Var\left(\int_\U f\,d\widetilde{\nu}^N_t\right) \le \sum_{j,k\in\Z} |\hat{f}(j)| |\hat{f}(k)| \left|\Cov_{\rho^N_t}\big(\tr[(\cdot)^j],\tr[(\cdot)^k]\big)\right| \le \frac{C(t)^2}{N^2}\left(\sum_{k\in\Z} |\hat{f}(k)|\right)^2.  \]
We can then estimate this squared-sum as in (\ref{e.thm1-2.sum}): writing $|\hat{f}(k)| = (1+k^2)^{-p/2}(1+k^2)^{p/2}|\hat{f}(k)|$,
\[ \left(\sum_{k\in\Z} |\hat{f}(k)|\right)^2 \le \left(\sum_{k\in\Z} (1+k^2)^{-p}\right)\cdot \|f\|_{H_p(\U)}^2, \]
and this sum is finite provided $p>\frac12$.  To summarize, if the conjectured bound (\ref{e.conc2?}) holds true, then we have
\begin{equation} \label{e.Var?} \Var\left(\int_\U f\,d\widetilde{\nu}^N_t\right) = O\left(\frac{1}{N^2}\right), \qquad \text{if} \;f\in H_p(\U)\; \text{for some} \; p>\textstyle{\frac12}. \end{equation}
In \cite[Theorem 2.6 \& Proposition 9.9]{Levy2010}, the authors showed that, if $f\in H_{1/2}(\U)$ is real-valued, then the fluctuations of the empirical integral are $O(1/N^2)$-Gaussian with variance close to $\|f\|_{H_{1/2}(\U)}^2$ for large $t$:
\[ N\left[\int_\U f\,d\widetilde{\nu}_t^N - \E\left(\int_\U f\,d\widetilde{\nu}_t^N\right)\right] \mathop{\longrightarrow}^{(d)}_{N\to\infty} \mathcal{N}(0,\sigma_t(f)), \qquad \lim_{t\to\infty} \sigma_t(f) = \|f\|_{H_{1/2}(\U)}^2. \]
We see from here that, at least as $t\to\infty$, we have $N^2\Var\left(\int_\U f\,d\widetilde{\nu}^N_t\right) \sim \|f\|_{H_{1/2}(\U)}^2$.  Thus, we cannot expect the conjectural $O(1/N^2)$-behavior of (\ref{e.Var?}) to hold for $f\notin H_{1/2}(\U)$, and so this is the minimal regularity needed for this rate of convergence.

\subsection{Empirical Eigenvalues on $\GL_N$\label{section empirical GL_N}}

We begin with the following observation: for {\em holomorphic} trace polynomials, $P\circ\mx{V}_N$ with $P\in\H\PP$, $\D_{s,t}$ reduces to $\D_{s-t,0}$.

\begin{lemma} \label{l.GLEig1} For $s,t>0$ with $s>t/2$, and for $P\in\H\PP$,
\begin{equation} \label{e.GLeig2} \begin{aligned} \big(e^{-\D^N_{s,t}}P\big)(\mx{1}) &= \big(e^{-\D^N_{s-t,0}}P\big)(\mx{1}), \\ \big(e^{-\D_{s,t}}P\big)(\mx{1}) &= \big(e^{-\D_{s-t,0}}P\big)(\mx{1}).
\end{aligned} \end{equation}
\end{lemma}

\begin{proof} For $P\in\H\PP$, the function $Z\mapsto P\circ\mx{V}_N(Z)$, $Z\in\GL_N$, is a trace polynomial in $Z$ and not $Z^\ast$; hence, it is holomorphic on $\GL_N$.  For any holomorphic function $f$ and any $X\in\u_N$,
\[ (\del_{iX} f)(Z) = \left.\frac{d}{dt}\right|_{t=0} f(Ze^{itX}) = i(\del_X f)(Z). \]
Hence $\del_{iX}^2 f = -\del_X^2 f$, and so (\ref{e.Ast}) yields
\[ A^N_{s,t}f = \left(s-\frac{t}{2}\right)\sum_{X\in\beta_N} \del_X^2 f + \frac{t}{2}\sum_{X\in\beta_N} \del_{iX}^2 f = (s-t)\sum_{X\in\beta_N}\del_X^2 f = (s-t)\Delta_{\U_N}f. \]
Applying the intertwining formulas (\ref{e.intertwine0}) and (\ref{e.intertwine0'}) now shows that
\[ \big(e^{-\D^N_{s,t}}P\big)\circ\mx{V}_N = \big(e^{-\D^N_{s-t,0}}P\big)\circ\mx{V}_N. \]
holds for all $N$.  Evaluating both sides at $I_N$ gives $\big(e^{-\D^N_{s,t}}P\big)(\mx{1}) = \big(e^{-\D^N_{s-t,0}}P\big)(\mx{1})$.  Taking the limit as $N\to\infty$ (using Corollary \ref{cor.conc}) now proves (\ref{e.GLeig2}).
\end{proof}

This brings us to the proof of Theorem \ref{thm 1.1}, which we break into two propositions.

\begin{proposition} \label{p.thm2-e1} Let $s,t>0$ with $s>t/2$.  Fix $\d>0$ and $f\in G_{\frac{s}{2}(1+2\d)}(\C^\ast)$.  Then
\begin{equation} \label{e.GLEig.Laur1} \left|\E\left(\int_{\C^\ast} f\,d\widetilde{\phi}_{s,t}^N\right) - \int f\,d\nu_{s-t}\right| \le \frac{1}{N^2}\cdot\frac{1}{\d}\left(1+\frac12\sqrt{\frac{\pi}{s\d}}\right)^{1/2}\|f\|_{G_{\frac{s}{2}(1+2\d)}}.\end{equation}
\end{proposition}

\begin{proof} The random variable $\int_{\C^\ast} f\,d\widetilde{\phi}^N_{s,t}$ is given by
\begin{equation} \label{e.emp.phi0} \left(\int_{\C^\ast} f\,d\widetilde{\phi}^N_{s,t}\right)(Z) = \sum_{k\in\Z} \hat{f}(k) \tr(Z^k), \qquad Z\in\GL_N, \end{equation}
which converges since, for any fixed $Z$, $|\tr(Z^k)|$ grows only exponentially in $k$, while by assumption $\hat{f}(k)$ decays super-exponentially fast. Note that
\begin{equation} \label{e.Eintphi0} \int f\,d\nu_{s-t} = \sum_{k\in\Z}  \hat{f}(k) \nu_k(s-t) = \sum_{k\in\Z}\hat{f}(k) \big(e^{-\D_{s-t,0}}v_k\big)(\mx{1}), \end{equation}
which converges as above since the $\nu_k(s-t)$ has only exponential growth.  Formally, we also have
\begin{align} \nonumber \E\left(\int_{\C^\ast} f\,d\widetilde{\phi}^N_{s,t}\right) &= \sum_{k\in\Z} \hat{f}(k) \int_{\GL_N} \tr(Z^k)\,\mu^N_{s,t}(dZ) \\
&= \sum_{k\in\Z}\hat{f}(k)\big(e^{-\D_{s,t}^N}v_k\big)(\mx{1}) = \sum_{k\in\Z} \hat{f}(k) \big(e^{-\D^N_{s-t,0}}v_k\big)(\mx{1}), \label{e.Eintphi1} \end{align}
by Lemma \ref{l.GLEig1}. The convergence of this series will follow from (\ref{e.GLEig.Laur1}), which we now proceed to prove.  Comparing (\ref{e.Eintphi0}) and (\ref{e.Eintphi1}),
\begin{equation} \label{e.Eintphi2} \left|\E\left(\int_{\C^\ast} f\,d\widetilde{\phi}^N_{s,t}\right)-\int f\,d\nu_{s-t}\right|\le \sum_{k\in\Z} |\hat{f}(k)| \left|\big(e^{-\D^N_{s-t,0}}v_k\big)(\mx{1})-\big(e^{-\D_{s-t,0}}v_k\big)(\mx{1})\right|. \end{equation}
We bound these terms using Corollary \ref{cor t.ind}: 
\[  \left|\big(e^{-\D^N_{s-t,0}}v_k\big)(\mx{1})-\big(e^{-\D_{s-t,0}}v_k\big)(\mx{1})\right| \le \frac{1}{N^2}\cdot\frac{1}{\d}e^{\frac{s}{2}(1+\d)k^2}\|v_k\|_1, \]
which holds true whenever $N>\sqrt{2/\d}$; note also that $\|v_k\|_1=1$. Thus (\ref{e.Eintphi2}) implies that
\[ \left|\E\left(\int_{\C^\ast} f\,d\widetilde{\phi}^N_{s,t}\right)-\int f\,d\nu_{s-t}\right|\le \frac{1}{N^2}\cdot\frac{1}{\d}\sum_{k\in\Z} |\hat{f}(k)|e^{\frac{s}{2}(1+\d)k^2}, \]
and this sum is bounded by
\begin{align} \nonumber \sum_{k\in\Z} e^{-\frac{s}{2}\d k^2} |\hat{f}(k)|e^{\frac{s}{2}(1+\d)k^2} &\le \left(\sum_{k\in\Z} e^{-s\d k^2}\right)^{\frac12} \|f\|_{G_{\frac{s}{2}(1+2\d)}} \\  \label{e.GaussIntBound}
&\le \left(1+\frac12\sqrt{\frac{\pi}{s\d}}\right)^{\frac12}\|f\|_{G_{\frac{s}{2}(1+2\d)}} \end{align}
where we have made the estimate
\[ \sum_{j=-n}^n e^{-s\d j^2} \le 1+2\int_0^\infty e^{-s\d x^2}\,dx = 1+\frac12\sqrt{\frac{\pi}{s\d}}. \]
This proves (\ref{e.GLEig.Laur1}).  \end{proof}

\begin{remark} In (\ref{e.Eintphi2}), we have used Lemma \ref{l.GLEig1} to convert $\D^N_{s-t,0}$ and $\D_{s-t,0}$ back to $\D^N_{s,t}$ and $\D_{s,t}$ to apply Corollary \ref{cor t.ind}.  We could instead have used that corollary with $r=|s-t|$ (or $r=\e$ for some $\e>0$ in the case $s=t$) to show the same result with the milder assumption that $f\in G_{\frac{r}{2}(1+2\d)}$.  This is not possible in Proposition \ref{p.thm2-e2} below where covariances are used, thus destroying the holomorphic structure; we have kept the regularity conditions consistent between the two.  \end{remark}

\begin{proposition} \label{p.thm2-e2} Let $s,t>0$ with $s>t/2$.  Fix $\d>0$ and $f\in G_{s(1+2\d)}(\C^\ast)$.  Then, for $N>\sqrt{2/\d}$,
\begin{equation} \label{e.GLEig.Laur2} \Var\left(\int_{\C^\ast} f\,d\widetilde{\phi}_{s,t}^N\right)  \le \frac{1}{N^2}\cdot\frac{4}{\d^2}\left(1+\frac12\sqrt{\frac{\pi}{2s\d}}\right)\|f\|_{G_{s(1+2\d)}}^2.\end{equation}
\end{proposition}

\begin{proof} Starting from (\ref{e.emp.phi0}), we expand
\begin{equation} \label{e.p.3.11.1} \Var\left(\int_{\C^\ast} f\,d\widetilde{\phi}^N_{s,t}\right) = \sum_{j,k\in\Z} \hat{f}(j)\overline{\hat{f}(k)} \Cov_{\mu^N_{s,t}}\big(\tr[(\cdot)^j],\tr[(\cdot)^k]\big). \end{equation}
Note that $\tr(Z^k) = v_k\circ\mx{V}_N(Z)$.  Since $v_k\in\PP_{|k|}$ and $\|v_k\|_1=1$, Corollary \ref{c.C_2} shows that
\begin{equation} \label{e.p.3.11.1.conv} \left|\Cov_{\mu^N_{s,t}}(v_j\circ\mx{V}_N,v_k\circ\mx{V}_N)\right| \le \frac{1}{N^2}\cdot\frac{4}{\d^2} e^{s(1+\d)(j^2+k^2)}. \end{equation}
Combining this with (\ref{e.p.3.11.1}) yields
\begin{align} \nonumber  \Var\left(\int_{\C^\ast} f\,d\widetilde{\phi}^N_{s,t}\right) &\le \frac{1}{N^2}\cdot\frac{4}{\d^2}\sum_{j,k\in\Z} |\hat{f}(j)||\hat{f}(k)|e^{s(1+\d)(j^2+k^2)} \\ \label{e.p.3.11.2}
&= \frac{1}{N^2}\cdot\frac{4}{\d^2}\left(\sum_{k\in\Z} |\hat{f}(k)|e^{s(1+\d)k^2}\right)^2, \end{align}
and the result follows from (\ref{e.GaussIntBound}) with $r$ replaced by $2s$.
\end{proof}

Thus, we have the ingredients to prove Theorem \ref{thm 1.1}.

\begin{proof}[Proof of Theorem \ref{thm 1.1}] \label{proof of thm 1.1} Since $\sigma>s$,  $\d = \frac12(\frac{\sigma}{s}-1)>0$ and $s(1+2\d)=\sigma$.  Thus Proposition \ref{p.thm2-e1} proves (\ref{e.holconv1}) with a constant that depends continuously on $s$ (note here that $G_{\frac{s}{2}(1+2\d)} = G_{\sigma/2} \subset G_\sigma$), and Proposition \ref{p.thm2-e2} similarly proves (\ref{e.holconv2}).  \end{proof}

\subsection{Empirical Singular Values on $\GL_N$\label{section empirical MN>0}}

As in Section \ref{section empirical GL_N}, we begin by noting a reduction in the action of the generator $\D_{s,t}$ of the noncommutative distribution $\p_{s,t}$ when restricted, in this case, to holomorphic trace polynomials in $ZZ^\ast$.  For this we need some new notation.

\begin{notation} For $k\in\Z\setminus\{0\}$, let $\ex^{1\ast}(k)=({\overset{2k}{\overbrace{1,\ast,\ldots,1,\ast}}})$ if $k>0$ and $\ex^{1\ast}(k)=({\overset{2|k|}{\overbrace{-1,-\ast,\ldots,-1,-\ast}}})$ if $k<0$; note that $|\ex^{1\ast}(k)| = 2|k|$. Denote $\EX^{1\ast} = \{\ex^{1\ast}(k)\colon k\in\Z\setminus\{0\}\}\subset\EX$.  Let $\PP^{1\ast}\subset\PP$  be the subalgebra of polynomials
\[ \PP^{1\ast} = \C\left[\{v_\ex\}_{\ex\in\EX^{1\ast}}\right]. \]
For convenience denote $v^{1\ast}_k = v_{\ex^{1\ast}(k)}$.

The homomorphism $\Phi_{1\ast}\colon \H\PP\to\PP^{1\ast}$ determined by $\Phi_{1\ast}(v_k) = v^{1\ast}_k$ is an algebra isomorphism.  Note that, for any $N\in\N$,
\begin{equation} \label{e.d.Phi1*} \Phi_{1\ast}(P)\circ\mx{V}_N = (P\circ\mx{V}_N)\circ\Phi, \end{equation}
where $\Phi(Z)=ZZ^\ast$ is the map from Definition \ref{d.eta}.
\end{notation}

\begin{lemma} \label{l.Phi1*} Let $s,t>0$ with $s>t/2$.  For $P\in\H\PP$ and $N\in\N$,
\begin{equation} \label{e.AstN1*} \begin{aligned} \big(e^{-\D^N_{s,t}} \Phi_{1\ast}(P)\big)(\mx{1}) &= \big(e^{\D^N_{t,0}}P\big)(\mx{1}), \\ \big(e^{-\D_{s,t}} \Phi_{1\ast}(P)\big)(\mx{1}) &= \big(e^{\D_{t,0}}P\big)(\mx{1}). \end{aligned}  \end{equation}
\end{lemma}

\begin{proof} For $Z\in\GL_N$ and $X\in\u_N$, note that
\begin{equation} \label{e.ZZ*diff0} Ze^{tX}(Ze^{tX})^\ast = Ze^{tX}e^{-tX}Z^\ast = ZZ^\ast, \qquad Ze^{itX}(Ze^{itX})^\ast = Ze^{2itX}Z^\ast. \end{equation}
Let $f\colon\GL_N\to\C$ be holomorphic.  The first equation in (\ref{e.ZZ*diff0}) shows that $\del_X (f\circ\Phi)= 0$, and so in particular the first terms $\sum_{X\in\beta_N} \del_X^2 (f\circ\Phi)=0$ in (\ref{e.Ast}).  For the second terms,
\begin{equation*} \del_{iX}\big(f\circ\Phi\big)(Z) = \left.\frac{d}{dt}\right|_{t=0} f\big(Ze^{itX}(Ze^{itX})^\ast\big) =  \left.\frac{d}{dt}\right|_{t=0} f(Ze^{2itX}Z^\ast), \end{equation*}
and so
\[  \del_{iX}^2\big(f\circ\Phi\big)(Z) = \left.\frac{\del^2}{\del s\del t}\right|_{s,t=0} f\big(Ze^{isX}e^{2itX}e^{isX}Z^\ast\big) = \left.\frac{\del^2}{\del s\del t}\right|_{s,t=0} f\big(Ze^{2i(s+t)X}Z^\ast\big). \]
If we additionally assume that $f$ is tracial, $f(ZW)=f(WZ)$ for all $Z,W\in\GL_N$ (for example if $f$ is a holomorphic trace polynomial $f=P\circ\mx{V}_N$ for some $P\in\H\PP$), then
\begin{equation} \label{e.ZZ*diff2}  \del_{iX}^2\big(f\circ\Phi\big)(Z) =  \left.\frac{\del^2}{\del s\del t}\right|_{s,t=0} f(Z^\ast Z e^{2i(s+t)X}) = 4\left.\frac{\del^2}{\del s\del t}\right|_{s,t=0} f(Z^\ast Z e^{i(s+t)X}). \end{equation}
By comparison,
\[ -\del_X^2 f(Z) = \del_{iX}^2 f(Z) = \left.\frac{\del^2}{\del s\del t}\right|_{s,t=0} f(Ze^{i(s+t)X}) \]
and so we have
\begin{equation} \del_{iX}^2(f\circ\Phi) = -4\big(\del_{X}^2 f\big)\circ \Phi^\perp \end{equation}
where $\Phi^\perp(Z) = Z^\ast Z$.  Hence, from (\ref{e.Ast}), we have
\begin{equation} \label{e.ZZ*diff3} \frac12A^N_{s,t} (f\circ\Phi) = \frac{t}{4}\sum_{X\in\beta_N}-4\big(\del_{X}^2 f\big)\circ\Phi^\perp = -t\big(\Delta_{\U_N} f\big)\circ\Phi^\perp = -t\big(\Delta_{\U_N} f\big)\circ\Phi, \end{equation}
where the last equality comes from the fact that $\Delta_{\U_N}$ preserves the class of smooth tracial functions.  (This follows from its bi-invariance, so it commutes with the left- and right-actions of the group; in our case, where $f$ will be a trace polynomial, it follows from the fact that $\Delta_{\U_N}$ preserves trace polynomials.)  Hence, taking $f=P\circ\mx{V}_N$ for some $P\in\H\PP$ and using (\ref{e.intertwine0}), (\ref{e.intertwine0'}), and (\ref{e.d.Phi1*}), we have
\[ \big(\D^N_{s,t}\Phi_{1\ast}(P)\big)\circ\mx{V}_N= -t\Phi_{1\ast}(\D^N_{1,0}P)\circ\mx{V}_N = \Phi_{1\ast}(-\D^N_{t,0}P)\circ\mx{V}_N. \]
Since $\Phi_{1\ast}$ is an algebra isomorphism, it follows that
\begin{equation} \label{e.ZZ*diff4} \big(e^{-\D^N_{s,t}}\Phi_{1\ast}(P)\big)\circ\mx{V}_N = \Phi_{1\ast}\big(e^{\D^N_{t,0}}P\big)\circ\mx{V}_N. \end{equation}
Evaluating both sides at $I_N$ gives
\[ \big(e^{-\D^N_{s,t}}\Phi_{1\ast}(P)\big)(\mx{1}) = \Phi_{1\ast}\big(e^{\D^N_{t,0}}P\big)(\mx{1}) = \big(e^{\D^N_{t,0}}P\big)(\mx{1}),  \]
the last equality following from the general fact that $\big(\Phi_{1\ast}(Q)\big)(\mx{1}) = Q(\mx{1})$.  Now letting $N\to\infty$ proves the lemma.
\end{proof}

We now approach the proof of Theorem \ref{thm 1.2} as we did for Theorem \ref{thm 1.1}.  We begin by verifying (\ref{e.posconv1}).

\begin{proposition} \label{p.GLSing1} Let $s,t>0$ with $s>t/2$.  Fix $\d>0$, and let $f\in G_{\frac{s}{2}(1+2\d)}(\C^\ast)$.  Then
\begin{equation} \label{e.GLSing.1} \left|\E\left(\int_{\C^\ast} f\,d\widetilde{\eta}_{s,t}^N\right) - \int f\,d\nu_{-2t}\right| \le \frac{1}{N^2}\cdot\frac{1}{\d}\left(1+\frac12\sqrt{\frac{\pi}{s\d}}\right)^{\frac12}\|f\|_{G_{\frac{s}{2}(1+2\d)}}.\end{equation}
\end{proposition}

\begin{proof} The random variable $\int_{\C^\ast} f d\widetilde{\eta}^N_{s,t}$ is given by
\begin{equation} \label{e.inteta} \left(\int_{\C^\ast} f d\widetilde{\eta}^N_{s,t}\right) = \sum_{k\in\Z} \hat{f}(k)\tr(Y^k), \qquad Y\in\M_N^{>0}, \end{equation}
which converges since, for any fixed $Y$, $|\tr(Y^k)|$ grows only exponentially in $k$, while by assumption $\hat{f}(k)$ decays super-exponentially fast.  We also have
\begin{equation} \label{e.Einteta0} \int f\,d\nu_{-t} = \sum_{k\in\Z}  \hat{f}(k) \nu_k(-t) = \sum_{k\in\Z}\hat{f}(k) \big(e^{\D_{t,0}}v_k\big)(\mx{1}), \end{equation}
which converges as above since $\nu_k(-t)$ have only exponential growth (being the moments of a compactly-supported probability measure).  By definition, subject to convergence,
\begin{align} \nonumber \E\left(\int_{\C^\ast} f\,d\widetilde{\eta}^N_{s,t}\right) &= 
\sum_{k\in\Z}\hat{f}(k)\int_{\GL_N} \tr(\Phi(Z)^k)\,\mu_{s,t}^N(dZ) \nonumber \\
&= \sum_{k\in\Z}\hat{f}(k)\big(e^{-\D_{s,t}^N}\Phi_{1\ast}(v_k)\big)(\mx{1}) 
= \sum_{k\in\Z} \hat{f}(k) \big(e^{\D^N_{t,0}}v_k\big)(\mx{1}), \label{e.Einteta1} \end{align}
by (\ref{e.d.Phi1*}) and Lemma \ref{l.Phi1*}.   The convergence of this series will follow from (\ref{e.GLSing.1}), which we now proceed to prove.  Comparing (\ref{e.Einteta0}) and (\ref{e.Einteta1}),
\begin{equation} \label{e.Einteta2} \left|\E\left(\int_{\C^\ast} f\,d\widetilde{\eta}^N_{s,t}\right)-\int f\,d\nu_{-t}\right|\le \sum_{k\in\Z} |\hat{f}(k)| \left|\big(e^{\D^N_{t,0}}v_k\big)(\mx{1})-\big(e^{\D_{t,0}}v_k\big)(\mx{1})\right|. \end{equation}
The remainder of the proof proceeds exactly as in the proof of Proposition \ref{e.GLEig.Laur1}, following (\ref{e.Eintphi2}).  \end{proof}

\begin{proposition} \label{p.GLSing2} Let $s,t>0$ with $s>t/2$.  Fix $\d>0$ and $f\in G_{4s(1+2\d)}(\C^\ast)$.  Then, for $N>\sqrt{2/\d}$,
\begin{equation} \label{e.GLSing.2} \Var\left(\int_{\C^\ast} f\,d\widetilde{\eta}_{s,t}^N\right)  \le \frac{1}{N^2}\cdot\frac{4}{\d^2}\left(1+\frac12\sqrt{\frac{\pi}{8s\d}}\right)\|f\|_{G_{4s(1+2\d)}}^2.\end{equation}
\end{proposition}

\begin{proof} As in (\ref{e.p.3.11.1}), we begin by expanding the variance from (\ref{e.inteta}) as follows:
\begin{equation} \label{e.GLSing2.1}  \Var\left(\int_{\C^\ast} f\,d\widetilde{\eta}_{s,t}^N\right) = \sum_{j,k\in\Z} \hat{f}(k)\overline{\hat{f}(k)} \Cov_{\Phi_\ast(\mu^N_{s,t})}\big(\tr[(\cdot)^j],\tr[(\cdot)^k]\big). \end{equation}
By definition, for any random variables $F,G$ on $\M_N^{>0}$,
\[ \Cov_{\Phi_\ast(\mu^N_{s,t})}\big(F,G) = \Cov_{\mu^N_{s,t}}(F\circ\Phi,G\circ\Phi). \]
With $F(Y) = \tr(Y^k)$, we have $F\circ\Phi = (v_k\circ\mx{V}_N)\circ\Phi = \Phi_{1\ast}(v_k)\circ\mx{V}_N$ by (\ref{e.d.Phi1*}), and so the covariances in (\ref{e.GLSing2.1}) are
\[ \left|\Cov_{\mu^N_{s,t}}\big(v^{1\ast}_j\circ\mx{V}_N,v^{1\ast}_k\circ\mx{V}_N\big)\right| \le \frac{1}{N^2}\cdot\frac{4}{\d^2}e^{s(1+\d)((2j)^2+(2k)^2)} \]
by Corollary \ref{c.C_2}, since $\deg(v^{1\ast}_k) = 2|k|$.  The remainder of the proof follows exactly as in the proof of Proposition \ref{p.thm2-e2}, following (\ref{e.p.3.11.1.conv}).
\end{proof}

This brings us to the proof of Theorem \ref{thm 1.2}

\begin{proof}[Proof of Theorem \ref{thm 1.2}] \label{proof of thm 1.2} Since $\sigma>4s$, $\d=\frac12(\frac{\sigma}{4s}-1)>0$ and $4s(1+2\d)=\sigma$.  Thus Proposition \ref{p.GLSing1} proves (\ref{e.posconv1}) with a constant that depends continuously on $s$ (note here that $G_{\frac{s}{2}(1+2\d)} = G_{\sigma/8}\subset G_\sigma$).  Similarly, Proposition \ref{p.GLSing2} proves (\ref{e.posconv2}).  \end{proof}

\section{$L^p$ Convergence \label{section strong convergence}}

In this final section, we observe that the techniques developed in Section \ref{section concentration hkm} in fact yield, with little extra effort, convergence in a sense significantly stronger than those given in Theorems \ref{thm 1}--\ref{thm 2}.  We begin with a brief discussion of strong convergence.

\subsection{Strong Convergence and Noncommutative $L^p$-norms \label{section strong convergence NC Lp}} 

Let $\rho^N$ be a probability measure on $\M_N$.  Suppose that the noncommutative empirical distribution $\widetilde{\p}^N$ of $\rho^N$ has a almost-sure limit distribution $\p$, in the sense of Definition \ref{d.convdist}.  In other words, if $A_N$ is a random matrix with distribution $\rho^N$, we have $\p_{A_N}(f) \to \p(f)\; a.s.$ for all noncommutative polynomials $f\in\C\langle A,A^\ast\rangle$.  The following stronger form of convergence has significant applications in operator algebras.

\begin{definition}[Strong Convergence] \label{d.strongconv} For each $N$, let $\rho^N$ be a probability measure on $\M_N$, and let $A_N$ be a random matrix with distribution $\rho^N$.  Say that {\bf $A_N$ converges strongly} if it converges in distribution and in operator norm almost surely.  That is: there exists a $C^\ast$-probability space $(\A,\t)$, and an element $a\in\A$, such that, for any noncommutative polynomial $f\in\C\langle A,A^\ast\rangle$,
\begin{equation} \label{e.strongconv} \tr[f(A_N,A_N^\ast)] \to \t[f(a,a^\ast)] \; a.s. \qquad \text{and} \qquad \|f(A_N,A_N^\ast)\|_{\M_N}\to \|f(a,a^\ast)\|_{\A} \; a.s. \end{equation}
\end{definition}
Definition \ref{d.strongconv} naturally generalizes to multivariate noncommutative distributions.  In their seminal paper \cite{Haagerup2005}, Haagerup and Thorbj{\o}rnsen showed that if $\rho^N$ is (a finite product of) the $\mathrm{GUE}_N$ measure (\ref{e.expot1}), then the independent $\mathrm{GUE}_N$ random matrices with this distribution converge strongly.  More recently, in \cite{Collins2013}, the authors showed that strong convergence also holds for (finite products of) the Haar measure on $\U_N$.  Given our mantra that the heat kernel measure $\rho_t^N$ on $\U_N$ interpolates between these two ensembles, it is natural to ask whether the matrices $U^N_t$ also exhibit strong convergence.  By extension, we may also ask whether random matrices $Z^N_{s,t}$ also exhibit strong convergence (now that we have proved, in Theorem \ref{thm 2}, that they have an almost-sure limit distribution).

Note that, for any matrix $A\in\M_N$, $\|A\| = \lim_{q\to\infty} \big(\tr\left[(AA^\ast)^q\right]\big)^{1/2q}$; since $AA^\ast\in\M_N^{>0}$ this makes sense for all real $q>0$, but for convenience we may restrict $q$ to be an integer.  In fact, the same holds true in any {\em faithful} noncommutative $C^\ast$-probability space $(\A,\t)$:
\[ \|a\|_{\A} = \lim_{q\to\infty} \big(\t\big[(aa^\ast)^q\big]\big)^{1/2q}. \]
These are (limits of) the {\bf noncommutative $L^p$-norms} over $(\A,\t)$:
\begin{equation} \label{e.ncLp} \|a\|_{L^p(\A,\t)} \equiv \big(\t\big[(aa^\ast)^{p/2}\big]\big)^{1/p}. \end{equation}
$\|\cdot\|_{L^p(\A,\t)}$ is a norm on $\A$ for $p\ge 1$. In the case that $\A$ is a $W^\ast$-algebra, its completion $L^p(\A,\t)$ can be realized as a space of unbounded operators affiliated to $\A$ when $p<\infty$, while $L^\infty(\A,\t)=\A$.

The second statement in (\ref{e.strongconv}) can thus be rephrased as an almost sure interchange of limits: since $(\M_N,\tr)$ is a faithful $C^\ast$-probability space, then $A_N\in\M_N$ converges to $a\in\A$ strongly if and only if $\p_{A_N}\to \p_a$ a.s. and
\begin{equation} \label{e.strongconv2} \P\left(\lim_{N\to\infty}\lim_{p\to\infty} \|f(A_N,A_N^\ast)\|_{L^p(\M_N,\tr)} = \lim_{p\to\infty} \|f(a,a^\ast)\|_{L^p(\A,\t)}\right)=1, \end{equation}
provided that $(\A,\t)$ is a faithful $C^\ast$-probability space.

\subsection{Almost Sure $L^p$ Convergence}

Theorem \ref{thm 2} establishes that the random matrices $U^N_t$ and $Z^N_{s,t}$ converge weakly almost surely to limit noncommutative distributions.  Indeed, the $U^N_t$ case (of convergence in expectation) is the main theorem in \cite{Biane1997c}, where it is shown that, if $U^N_t$ is chosen to be a Brownian motion on $\U_N$, then the weak limit exists as a noncommutative stochastic process, the free unitary Brownian motion discussed at the end of Section \ref{section free prob}.  In this case, the limit noncommutative probability space can be taken as a free group factor, and so is indeed a faithful $C^\ast$-probability space.  As for $Z^N_{s,t}$, Definition \ref{d.phist} and the subsequent discussion show how to realize the almost sure limit noncommutative distribution $\p_{s,t}$ as the distribution of an operator $\p_{s,t}=\p_{z_{s,t}}$ on a noncommutative probability space $(\A_{s,t},\t_{s,t})$ (although we have not yet been able to establish that $\t_{s,t}$ is faithful).  As such, we can construct a larger $C^\ast$-probability space that contains both of the limit operators $u_t$ and $z_{s,t}$.  (By taking the reduced free product $C^\ast$-algebra of the two spaces, we can even make $u_t$ and $z_{s,t}$ freely independent if we wish.)  Thus, in the statement of Theorem \ref{thm 3}, there is no loss of generality in realizing the limits in a single $C^\ast$-probability space $(\A,\t)$.

While we are, as yet, unable to prove strong convergence of $U^N_t$ and $Z^N_{s,t}$ to $u_t$ and $z_{s,t}$, we can prove almost sure $L^p$-convergence for all even integers $p$, i.e.\ Theorem \ref{thm 3}. From (\ref{e.strongconv2}), this result should be viewed as only infinitesimally weaker.  Once again, they key is a variance estimate, which follows easily from Proposition \ref{p.thm2.0}.  

\begin{lemma} \label{l.final} Let $s,t>0$ with $s>t/2$, and let $f\in\C\langle A,A^\ast\rangle$ be a noncommutative polynomial. Let $p\ge 2$ be an even integer.  Then, for $N\in\N$,
\[ \Var\Big(\|f\big(U^N_t,(U^N_t)^\ast\big)\|_{L^p(\M_N,\tr)}^p\Big) = O\left(\frac{1}{N^2}\right) \quad \text{and} \quad \Var\Big(\|f\big(Z^N_{s,t},(Z^N_{s,t})^\ast\big)\|_{L^p(\M_N,\tr)}^p\Big) = O\left(\frac{1}{N^2}\right). \]
\end{lemma}

\begin{proof} We begin with the case of $Z^N_{s,t}$.  The variance in question is
\begin{equation} \label{e.LpVar1} \Var\Big(\|f\big(Z^N_{s,t},(Z^N_{s,t})^\ast\big)\|_{L^p(\M_N,\tr)}^p\Big)  = \Var_{\mu^N_{s,t}}(F^p), \end{equation}
where $F^p\colon\GL_N\to\C$ is the random variable
\[ F^p(Z) = \|f(Z,Z^\ast)\|_{L^p(\M_N,\tr)}^p = \tr\left(\big(f(Z,Z^\ast)f(Z,Z^\ast)^\ast\big)^{p/2}\right). \]
Note that $g_p(A,A^\ast) = \big(f(A,A^\ast)f(A,A^\ast)^\ast\big)^{p/2}$ is an element of $\C\langle A,A^\ast\rangle$.  Thus using the inclusion $\Upsilon$ of $\C\langle A,A^\ast\rangle\hookrightarrow \PP^+$ (\ref{e.Upsilon}),  we have
\begin{equation} \label{e.LpVar2} F^p(Z) = \Upsilon(g_p)\circ\mx{V}_N(Z). \end{equation}
By Proposition \ref{p.thm2.0},
\begin{equation} \label{e.LpVar2.5} \Var_{\mu^N_{s,t}}(\Upsilon(g_p)\circ\mx{V}_N) \le \frac{1}{N^2}\cdot C_2(s,t,\Upsilon(g_p),\Upsilon(g_p)), \end{equation}
and this, together with (\ref{e.LpVar1}) and (\ref{e.LpVar2}), proves the lemma for $Z^N_{s,t}$.  The statement for $U^N_t$ actually follows as a special case.  Indeed, for any $P\in\PP$, (\ref{e.intertwineC}) and (\ref{e.hkm1'}) show that
\begin{equation} \label{e.LpVar3} \Var_{\rho^N_t}(P\circ\mx{V}_N) = \big(e^{-\D^N_{t,0}}(PP^\ast)\big)(\mx{1}) - \big(e^{-\D^N_{t,0}}P\big)(\mx{1}) \big(e^{-\D^N_{t,0}}P^\ast\big)(\mx{1}). \end{equation}
Proposition \ref{p.thm2.0} is proved by showing that this quantity, with $\D_{s,t}^N$ in place of $\D_{t,0}^N$, is $\le C_2(s,t,P,P)/N^2$.  Although we must have $s,t>0$ and $s>t/2$ for $\mu^N_{s,t}$ to be a well-defined measure, the operators $e^{-\D^N_{s,t}}$, and ergo the quantities in (\ref{e.LpVar3}) and the constant $C_2(s,t,P,P)$, are all well-defined for $s,t\in\R$.  Thus, we may restrict (\ref{e.LpVar2.5}) to find
\begin{equation} \label{e.LpVar3.5} \Var_{\rho^N_t}(F^p) = \Var_{\rho^N_t}(\Upsilon(g_p)\circ\mx{V}_N) \le \frac{1}{N^2}\cdot C_2(t,0,\Upsilon(g_p),\Upsilon(g_p)), \end{equation}
and this proves the $U^N_t$-case of the lemma.  \end{proof}

\begin{remark} \label{r.final-} The size of the constant $C_2(t,0,P,P)$ has only been shown (Corollary \ref{c.C_2}) to be bounded (almost) by $e^{2t\cdot\deg(P)^2}\|P\|_1^2$.  We conjecture (as in (\ref{e.conc3?})) that the growth with $\deg(P)$ is erroneous; but the dependence on $\|P\|_1$ is surely not.  It is relatively straightforward to calculate that, with $g_p$ defined from $f$ as in the proof of Lemma \ref{l.final}, 
\[ \|\Upsilon(g_p)\|_1 = \|\Upsilon(f)\|_1^p. \]
This is not unexpected, since the $L^p$-norm itself is the $p$th root of the quantities considered here.   \end{remark}


This brings us, finally, to the proof of Theorem \ref{thm 3}.

\begin{proof}[Proof of Theorem \ref{thm 3}] The almost sure weak convergence of $Z^N_{s,t}$ to $z_{s,t}$ was established in Theorem \ref{thm 2}; $U^N_t$ follows as the special case $Z^N_{t,0}$ (and was established already in \cite{Rains1997}).  It follows that, for any $f\in\C\langle A,A^\ast\rangle$, 
\begin{align*} &\E\left(\|f(U^N_t,(U^N_t)^\ast)\|^p_{L^p(\M_N,\tr)}\right) \to \|f(u_t,u_t^\ast)\|^p_{L^p(\A,\t)}, \quad \text{and} \\
&\E\left(\|f(Z^N_{s,t},(Z^N_{s,t})^\ast)\|^p_{L^p(\M_N,\tr)}\right) \to \|f(z_{s,t},z_{s,t}^\ast)\|^p_{L^p(\A,\t)}, \end{align*}
since these quantities (rased to the $p$th power as they are) are trace polynomials in $U_t^N$ (resp.\ $Z^N_{s,t}$) and $u_t$ (resp.\ $z_{s,t}$).  Lemma \ref{l.final}, together with Chebyshev's inequality and the Borel-Cantelli Lemma, now shows that
\begin{align*} &\|f(U^N_t,(U^N_t)^\ast)\|^p_{L^p(\M_N,\tr)} \to \|f(u_t,u_t^\ast)\|^p_{L^p(\A,\t)}\; a.s. \quad \text{and} \\
&\|f(Z^N_{s,t},(Z^N_{s,t})^\ast)\|^p_{L^p(\M_N,\tr)} \to \|f(z_{s,t},z_{s,t}^\ast)\|^p_{L^p(\A,\t)}\; a.s. \end{align*}
The theorem now follows by taking $p$th roots.
\end{proof}

\begin{remark} The above proof, coupled with Remark \ref{r.final-}, shows that it is plausible that the rate of a.s.\ convergence in Theorem \ref{thm 3} is uniformly bounded in $p$ (contingent on the conjectured trace degree-independence of the constants $C_2(t,0,P,P)$) in the $U_t^N$-case.  If this is true, then strong convergence $U^N_t\to u_t$ follows readily from (\ref{e.strongconv2}).  This is left as a promising avenue for future study.
\end{remark}

\subsection*{Acknowledgments} The author wishes to thank Bruce Driver for many helpful and insightful conversations, particularly with regards to Section \ref{section proof of thm 1}.

\bibliographystyle{acm}
\bibliography{usfb}

\end{document}